\documentclass[letterpaper,12pt]{article}
\usepackage{tikz}
\usepackage{graphicx}
\usepackage{color}
\usepackage{setspace}
\usepackage{url}
\usepackage{latexsym,amsmath,amssymb,theorem,graphics,epsfig,subfigure}
\usepackage{color}
\usepackage{multirow}
\theoremstyle{break}
\newtheorem{theorem}{Theorem}
\newtheorem{corollary}{Corollary}
\newtheorem{proposition}{Proposition}
\newtheorem{lemma}{Lemma}
\theorembodyfont{\rm}
\newtheorem{definition}{Definition}
\newtheorem{example}{Example}

\newtheorem{remark}{Remark}

\newcommand{\lran}{\longrightarrow}

\def\lint#1{\left[{#1}\right[}
\def\rint#1{\left]{#1}\right]}
\def\cint#1{\left[{#1}\right]}
\def\opint#1{\left]{#1}\right[}
\def\uint{[0,1]}

\def\ouint{\left]0,1\right[}

\def\prooftxt{\mbox{\large\sc proof: }}
\def\konproof{\rm\hspace*{\fill}$\Box$}
\newenvironment{proof}{\par\smallskip\noindent\prooftxt }%
{\konproof\par\vspace*{6pt}}
\newcounter{inst}
\newenvironment{mylist}{\begin{list}{(\roman{enumi})}{\usecounter{enumi}%
                        \setlength{\listparindent}{0pt}%
                        \setlength{\labelsep}{0.4em}%
                        \setlength{\labelwidth}{2.1em}
                        \setlength{\leftmargin}{2em}%
                        \setlength{\itemsep}{-0.8ex}%
                        \setlength{\topsep}{0.4ex}%
                        \setlength{\rightmargin}{0pt}}}{\end{list}}
\date{}
\author{
{\bf Andrea Mesiarov\' a-Zem\' ankov\' a} \\ \ Mathematical Institute, Slovak Academy of Sciences \\
Bratislava, SLOVAKIA \\
zemankova@mat.savba.sk
}

\title{\bf  Characterization of uninorms with continuous underlying t-norm and t-conorm by means of an extended ordinal sum}

\begin{document}

\maketitle \thispagestyle{empty}

\begin{abstract}  The uninorms with continuous underlying t-norm and t-conorm
 are characterized via an extended ordinal sum construction. Using the results of \cite{mulfun}, where each   uninorm with continuous underlying operations was characterized by properties of its set of discontinuity points, it is shown that each such a uninorm
 can be decomposed into an extended ordinal sum of representable uninorms, continuous Archimedean t-norms, continuous Archimedean t-conorms and internal uninorms.

{\bf Keywords:}  uninorm, representable uninorm, additive generator, t-norm, ordinal sum
\end{abstract}

\section{Introduction}
\label{sec1}

The (left-continuous) t-norms and their dual t-conorms have an indispensable role in many domains \cite{haj98,16,sug85}.
 Each continuous t-norm (t-conorm) can be expressed as an ordinal sum of continuous Archimedean t-norms (t-conorms), while each Archimedean t-norm (t-conorm) is generated by an additive generator (see \cite{AFS06,KMP00}).
 Generalizations of t-norms and t-conorms that can model bipolar behaviour are uninorms
 (see \cite{FYR97,jacon,YR96}). The class of uninorms is widely used both in theory \cite{PT1,uniru} and in applications \cite{PA2,PA1}.
 The complete characterization of uninorms with continuous underlying t-norm and t-conorm has been in the center of the interest for a long time, however, only partial results were achieved (see \cite{ordut,FB12,Lili, Lilif,MesPet14,Rut14}).

 In \cite{genuni} we have introduced ordinal sum of uninorms and in \cite{repord} we have characterized uninorms that are ordinal sums of representable uninorms. This paper is a continuation of the paper \cite{mulfun}, where we have characterized uninorms with continuous
 underlying operations by properties of their set of discontinuity points.
 Our aim is to completely characterize all uninorms with continuous underlying functions and obtain a similar representation as in the case of t-norms and t-conorms. In this paper we will therefore  show that  each uninorm with continuous underlying t-norm and t-conorm can be decomposed into an extended ordinal sum of representable uninorms, continuous Archimedean  t-norms, continuous Archimedean t-conorms and internal uninorms.
 Let us now recall all necessary basic notions.

 A triangular norm  is a binary function $T \colon \uint^2 \lran \uint $ which is commutative,
associative, non-decreasing in both variables and $1$ is its neutral element. Due to the associativity, $n$-ary form of any t-norm is uniquely given and  thus it can be extended to an aggregation function working on  $\bigcup_{n\in \mathbb{N}}\uint^n.$
Dual functions to t-norms are t-conorms. A triangular conorm  is a binary function $C \colon \uint^2 \lran \uint $ which is commutative,
associative, non-decreasing in both variables and $0$ is its neutral element.
The duality between t-norms and t-conorms is expressed by the fact that from any t-norm $T$ we can obtain its dual t-conorm $C$ by the equation
$$C(x,y)=1-T(1-x,1-y)$$
and vice-versa.

\begin{proposition}
Let $t\colon \uint\lran \cint{0,\infty}$ ($c\colon \uint\lran \cint{0,\infty}$) be a continuous strictly decreasing (increasing) function such that $t(1)=0$ ($c(0)=0$). Then the binary  operation  $T \colon \uint^2 \lran \uint $ ( $C \colon \uint^2 \lran \uint $)  given by
$$T(x,y)=t^{-1}(\min(t(0),t(x)+t(y)))$$
$$C(x,y)=c^{-1}(\min(c(1),c(x)+c(y)))$$
is a continuous t-norm (t-conorm). The function $t$ ($c$) is called an \emph{additive generator} of $T$ ($C$).
\end{proposition}

An additive generator of a continuous t-norm $T$ (t-conorm $C$) is uniquely determined up to a positive multiplicative constant.

Now let us recall an ordinal sum construction for t-norms and t-conorms \cite{KMP00}.
However, first we recall an original definition of an ordinal sum of semigroups by Clifford \cite{cli}.

\begin{theorem}
Let $A\neq \emptyset$ be a totally ordered set and $(G_{\alpha})_{\alpha\in A}$ with $G_{\alpha}=(X_{\alpha},*_{\alpha})$ be a family of semigroups.
Assume that for all
$\alpha,\beta\in A$ with
$\alpha<\beta$ the sets $X_{\alpha}$
 and $X_{\beta}$ are either disjoint or that $X_{\alpha} \cap X_{\beta}=\{x_{\alpha,\beta}\},$  where
$x_{\alpha,\beta}$ is both the neutral element of $G_{\alpha}$
 and the annihilator of $G_{\beta}$ and where for each $\gamma\in A$ with
$\alpha<\gamma<\beta$ we
have $X_{\gamma}=\{x_{\alpha,\beta}\}.$ Put $X=\bigcup\limits_{\alpha\in A}X_{\alpha}$
 and define the binary operation $*$ on $X$ by
$$x*y=\begin{cases} x*_{\alpha} y &\text{if $(x,y)\in X_{\alpha} \times X_{\alpha},$} \\
x &\text{if $(x,y)\in X_{\alpha} \times X_{\beta}$ and $\alpha<\beta,$} \\
y &\text{if $(x,y)\in X_{\alpha} \times X_{\beta}$ and $\alpha>\beta.$} \\ \end{cases}$$
Then $G =(X,*)$ is a semigroup. The semigroup $G$ is commutative if and only if for each
 $\alpha \in A$ the semigroup
$G_{\alpha}$
 is commutative. \label{thcli}
\end{theorem}

\begin{proposition} Let $K$ be a finite or countably infinite index set and
let $(\opint{a_k,b_k})_{k\in K}$ ($(\opint{c_k,d_k})_{k\in K}$) be a disjoint system of open subintervals of $\uint.$ Let $(T_k)_{k\in K}$ ($(C_k)_{k\in K}$) be
a system of  t-norms (t-conorms). Then
the ordinal sum $T =(\langle a_k, b_k,T_k \rangle \mid  k \in K)$ ($C =(\langle a_k, b_k,C_k \rangle \mid  k \in K)$) given by
$$T(x,y)=\begin{cases} a_k +(b_k-a_k) T_k(\frac{x-a_k}{b_k-a_k},\frac{y-a_k}{b_k-a_k}) &\text{if $(x,y)\in \lint{a_k,b_k}^2 ,$} \\
\min(x,y) &\text{else} \end{cases}$$
and
$$C(x,y)=\begin{cases} c_k +(d_k-c_k) C_k(\frac{x-c_k}{d_k-c_k},\frac{y-c_k}{d_k-c_k}) &\text{if $(x,y)\in \rint{c_k,d_k}^2 ,$} \\
\max(x,y) &\text{else} \end{cases}$$
is a  t-norm (t-conorm). The t-norm $T$ (t-conorm $C$) is continuous if and only if all summands $T_k$ ($C_k$) for $k\in K$ are continuous.
\end{proposition}

Each continuous t-norm (t-conorm) is equal to an ordinal sum of continuous Archimedean  t-norms (t-conorms). Note that a continuous t-norm (t-conorm) is Archimedean if and only if it has only trivial idempotent points $0$ and $1.$ A continuous Archimedean t-norm $T$ (t-conorm $C$) is either strict, i.e.,
strictly increasing on $\rint{0,1}^2$ (on $\lint{0,1}^2$), or nilpotent, i.e.,  there exists $(x,y)\in \opint{0,1}^2$ such that $T(x,y)=0$
($C(x,y)=1$). Moreover, each continuous Archimedean t-norm (t-conorm) has a continuous additive generator. More details on t-norms and t-conorms can be found in \cite{AFS06,KMP00}.

A uninorm  (introduced in \cite{YR96}) is a binary function  $U\colon \uint^2 \lran \uint $  which is commutative,
associative, non-decreasing in both variables and have a neutral element $e\in \ouint$ (see also \cite{FYR97}). If we take uninorm in a broader sense, i.e., if for a neutral element we have  $e\in \uint,$ then the class of uninorms covers also the class of t-norms and the class of t-conorms. In order the stress that we assume a uninorm with $e\in \opint{0,1}$ we will call such a uninorm \emph{proper}. For each uninorm the value $U(1,0)\in \{0,1\}$ is the annihilator of $U.$ A uninorm is called \emph{conjunctive} (\emph{disjunctive}) if $U(1,0)=0$ ($U(1,0)=1$). Due to the associativity we can uniquely define $n$-ary form of any uninorm for any $n\in\mathbb{N}$ and therefore in some proofs we will use ternary form instead of binary, where suitable.

For each uninorm $U$ with the  neutral element $e\in \uint,$ the restriction of $U$ to $\cint{0,e}^2$ is a t-norm on $\cint{0,e}^2,$ i.e., a linear transformation of some t-norm $T_U$  on $\uint^2$
and the restriction of $U$ to $\cint{e,1}^2$ is a t-conorm  on $\cint{e,1}^2,$ i.e., a linear transformation of some t-conorm $C_U.$  Moreover, $\min(x,y)\leq U(x,y)\leq \max(x,y)$ for all
$(x,y)\in \cint{0,e}\times \cint{e,1}\cup \cint{e,1}\times \cint{0,e}.$

From any pair of a t-norm and a t-conorm we can construct the minimal and the maximal uninorm with the given
underlying functions.

\begin{proposition}
Let  $T \colon \uint^2 \lran \uint $ be a t-norm and  $C \colon \uint^2 \lran \uint $ a t-conorm and assume $e\in \uint.$ Then the two functions  $U_{\min},U_{\max} \colon \uint^2 \lran \uint $ given by
$$U_{\min}(x,y)=\begin{cases} e\cdot T(\frac{x}{e},\frac{y}{e}) &\text{if $(x,y)\in \cint{0,e}^2,$ } \\
e+ (1-e)\cdot C(\frac{x-e}{1-e},\frac{y-e}{1-e}) &\text{if $(x,y)\in \cint{e,1}^2,$ } \\
\min(x,y) &\text{otherwise}\end{cases}$$ and
$$U_{\max}(x,y)=\begin{cases} e\cdot T(\frac{x}{e},\frac{y}{e}) &\text{if $(x,y)\in \cint{0,e}^2,$ } \\
e+ (1-e)\cdot C(\frac{x-e}{1-e},\frac{y-e}{1-e}) &\text{if $(x,y)\in \cint{e,1}^2,$ } \\
\max(x,y) &\text{otherwise}\end{cases}$$ are uninorms. We will denote the set of all uninorms of the first type by $\mathcal{U}_{\min}$ and of the second type by $\mathcal{U}_{\max}.$
\end{proposition}

\begin{definition}
A uninorm  $U \colon \uint^2 \lran \uint $ is called
\emph{internal} if  $U(x,y)\in \{x,y\}$ for all $(x,y)\in \uint^2.$ Moreover, $U$ is called s-internal if it is internal and there exists a continuous and strictly decreasing function $v_{U}\colon \uint \lran \uint$ such that  $U(x,y)=\min(x,y)$ if $y<v_U(x)$ and  $U(x,y)=\max(x,y)$ if $y>v_U(x).$ Finally, $U$ is called pseudo-internal if $U$ is internal on $\cint{0,e}\times \cint{e,1}\cup \cint{e,1}\times \cint{0,e}.$
\end{definition}

For example all uninorms from $\mathcal{U}_{\min} \cup \mathcal{U}_{\max}$ are pseudo-internal. 

The following easy lemma was shown in \cite{genuni}.

\begin{lemma}
Let $U\colon \uint^2 \lran \uint $  be a uninorm such that $T_U=\min$ and $C_U=\max.$ Then $U$ is internal. \label{lemmm}
\end{lemma}

For a given uninorm $U\colon \uint^2\lran \uint$ and  each $x\in \uint$ we define a function $u_x\colon \uint\lran \uint$ by $u_x(z)=U(x,z)$ for $z\in \uint.$

Similarly as in the case of t-norms and t-conorms we can construct uninorms using additive generators (see  \cite{FYR97}).

\begin{proposition}
Let $f\colon \uint\lran\cint{-\infty,\infty},$  $f(0)=-\infty,$ $f(1)=\infty$ be a continuous strictly increasing function.
Then a binary function $U \colon \uint^2 \lran \uint $ given by
$$U(x,y)=f^{-1}(f(x)+f(y)),$$ where $f^{-1}\colon \cint{-\infty,\infty}\lran \uint$ is an inverse function to $f,$ is a uninorm, which will be called a \emph{representable} uninorm.
\end{proposition}

Note that if we relax the monotonicity of the additive generator then the neutral element will be lost and by relaxing the condition  $f(0)=-\infty,$ $f(1)=\infty$ the associativity will be lost (if $f(0)<0$ and $f(1)>0$). In  \cite{uniru} (see also \cite{jacon}) we can find the following result.

\begin{proposition}
Let $U\colon \uint^2 \lran \uint$ be a uninorm continuous everywhere on the unit square except of the two points $(0,1)$ and $(1,0).$ Then $U$ is representable.
\label{prouni}
\end{proposition}

Thus a uninorm $U$ is representable if and only if it is continuous on $\uint^2\setminus \{(0,1),(1,0)\},$ which
 completely characterizes the set of representable uninorms.

 \begin{definition}
We will denote the set of all uninorms $U$ such that $T_U$ and $C_U$ are continuous by $\mathcal{U},$ and the set of all uninorms $V$ such that
$V(x,0)=0$ for all $x\in \lint{0,1}$  and $V(x,1)=1$ for all $x\in \rint{0,1}$  by $\mathcal{N}.$ Further, we will denote by $\mathcal{N}_{\max}$ ($\mathcal{N}_{\min}$) the set of all uninorms $U\in \mathcal{N}$ such that there exists a uninorm $U_1\in \mathcal{U}_{\max}$ ($U_1\in \mathcal{U}_{\min}$) such that
$U=U_1$ on $\opint{0,1}^2.$
\end{definition}

An ordinal sum of uninorms was introduced in \cite{genuni}.
For any $0\leq a\leq b<c\leq d\leq 1,$  $v\in \cint{b,c},$ and a uninorm $U$  with the neutral element $e\in \uint$ we will use a  transformation $f\colon \uint \lran \lint{a,b} \cup \{v\} \cup \rint{c,d}$
 given by
\begin{equation}f(x)=\begin{cases}  (b-a)\cdot \frac{x}{e} + a  &\text{if $x\in \lint{0,e} ,$ } \\
v  &\text{if $x=e ,$ } \\
d - \frac{(1-x)(d-c)}{(1-e)} &\text{otherwise.} \end{cases}\label{utra}\end{equation}
Then $f$  is linear on $\lint{0,e}$ and on $\rint{e,1}$ and thus it is a
piece-wise linear isomorphism of $\uint$ to $(\lint{a,b} \cup \{v\} \cup \rint{c,d})$  and a binary function
$U^{a,b,c,d}_v\colon (\lint{a,b} \cup \{v\} \cup \rint{c,d})^2 \lran (\lint{a,b} \cup \{v\} \cup \rint{c,d})$ given by
\begin{equation}U^{a,b,c,d}_v(x,y)=f(U(f^{-1}(x),f^{-1}(y)))\label{unitr}\end{equation} is a uninorm on $(\lint{a,b} \cup \{v\} \cup \rint{c,d})^2.$ Note that in the case
when $a=b$ ($c=d$) we will transform only the part of the uninorm $U$ which is defined on  $\cint{e,1}^2$ ($\cint{0,e}^2$).
The function $f$ is piece-wise linear, however, more generally we can use any increasing isomorphic transformation.

\begin{remark}
If $U_1$ and $U_2$ are uninorms, with respective neutral elements $e_1,e_2,$ then for $0\leq a<  b<c< d \leq 1$ we have $$(U_1)^{a,b,c,d}_v(x,y)=(U_2)^{a,b,c,d}_v(x,y)$$
if and only if $U_1(x,y)=\phi^{-1}(U_2(\phi(x),\phi(y))),$ where $\phi\colon \uint\lran \uint$ is a strictly increasing isomorphism with $\phi(e_1)=e_2$
which is linear on $\cint{0,e_1}$ and on $\cint{e_1,1}.$ Similar result can be obtain for the case when  $a=b$ ($c=d$), however, then only the corresponding parts of uninorms are isomorphic.
\end{remark}

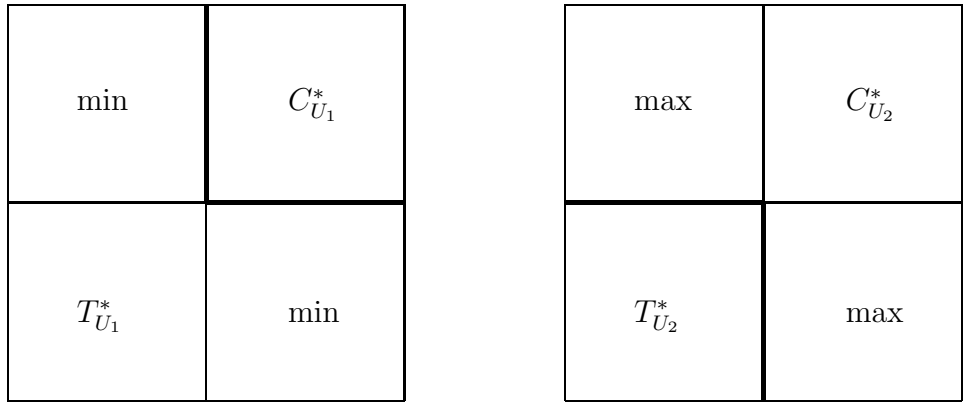
\begin{figure}[b]  \hskip2cm
\begin{picture}(150,150)
\put(0,0){ \line(1,0){150} }
\put(0,0){ \line(0,1){150} }
\put(150,150){ \line(-1,0){150} }
\put(150,150){ \line(0,-1){150} }
\put(0,75){ \line(1,0){150} }
\put(75,0){ \line(0,1){150} }
\linethickness{0.5mm}
\put(75,75){ \line(1,0){75} }
\put(75,75){ \line(0,1){75} }

\put(30,30){$T_{U_1}^*$}
\put(110,110){$C_{U_1}^*$}
\put(110,30){$\min$}
\put(30,110){$\min$}
\end{picture}
\hskip2cm
\begin{picture}(150,150)
\put(0,0){ \line(1,0){150} }
\put(0,0){ \line(0,1){150} }
\put(150,150){ \line(-1,0){150} }
\put(150,150){ \line(0,-1){150} }
\put(0,75){ \line(1,0){150} }
\put(75,0){ \line(0,1){150} }
\linethickness{0.5mm}
\put(0,75){ \line(1,0){75} }
\put(75,0){ \line(0,1){75} }
\put(30,30){$T_{U_2}^*$}
\put(110,110){$C_{U_2}^*$}
\put(110,30){$\max$}
\put(30,110){$\max$}
\end{picture}
\caption{The uninorm $U_1$ (left) and the uninorm $U_2$ (right) from Example \ref{exmm}.  The bold lines denote the points of discontinuity of $U_1$ and $U_2.$} \label{figst1}
\end{figure}

\begin{proposition}
Assume $e\in \uint.$
Let $K$ be an index set which is finite or countably infinite and let $(\opint{a_k,b_k})_{k\in K}$   be a  disjoint system of open subintervals (which can be also empty) of $\cint{0,e},$ such that $\bigcup_{k\in K}\cint{a_k,b_k}=\cint{0,e}.$ Similarly, let
 $(\opint{c_k,d_k})_{k\in K}$ be a disjoint system of open subintervals (which can be also empty) of $\cint{e,1},$ such that $\bigcup_{k\in K}\cint{c_k,d_k}=\cint{e,1}.$ Let further these two systems be anti-comonotone, i.e., $b_k\leq a_i$ if and only if  $c_k \geq d_i$
  for all $i,k\in K.$
  Assume
a family of   uninorms  $(U_k)_{k\in K}$ on $\uint^2$ such that if both  $\opint{a_k,b_k}$ and $\opint{c_k,d_k}$ are non-empty then $U_k$ is a proper uninorm, if $\opint{a_k,b_k}$ is non-empty $U_k$ is either a t-norm or a proper uninorm and  if $\opint{c_k,d_k}$ is non-empty then $U_k$ is either a t-conorm or a proper uninorm, and finally if both $\opint{a_k,b_k}$ and $\opint{c_k,d_k}$ are empty
then $a_k=b_k=a_{k_1}=b_{k_1}$ and $c_k=d_k=c_{k_1}=d_{k_1}$ does not hold for any $k_1\in K,$ $k\neq k_1,$ and here
 only the value $U_k(0,1)$ is interesting.
 Denote $B=\{b_k\mid k\in K\}\setminus \{a_k\mid k\in K\}$ and $C=\{c_k\mid k\in K\}\setminus \{d_k\mid k\in K\}.$ We define a function
  $n\colon B\lran B\cup C$ given for all $b_k\in B$ by
  $$n(b_k)=\begin{cases} b_k &\text{if $U_k(1,0)=0,$} \\ c_k &\text{else.} \end{cases}$$
Let the ordinal sum $U^e=(\langle a_k,b_k,c_k,d_k,U_k \rangle \mid k\in K)^e$ be given by
$$U^e(x,y)=\begin{cases}
y &\text{if  $x=e,$  }\\
x &\text{if  $y=e,$  }\\
(U_k)^{a_k,b_k,c_k,d_k}_{v_k} &\text{if  $(x,y)\in (\lint{a_k,b_k}\cup \rint{c_k,d_k})^2,$ }\\
x &\text{if $y\in \cint{b_k,c_k}, x\in \cint{a_k,d_k}\setminus\cint{b_k,c_k},$} \\
y &\text{if $x\in \cint{b_k,c_k}, y\in \cint{a_k,d_k}\setminus\cint{b_k,c_k},$} \\
\min(x,y) &\text{if $(x,y)\in\cint{b_k,c_k}^2\setminus (\opint{b_k,c_k}^2\cup \{(b_k,c_k),(c_k,b_k)\}),$}\\
&\text{ where $b_k\in B,$ $c_k\in C,$ $x+y<c_k+b_k,$}\\
\max(x,y) &\text{if $(x,y)\in\cint{b_k,c_k}^2\setminus (\opint{b_k,c_k}^2\cup \{(b_k,c_k),(c_k,b_k)\}),$}\\
&\text{ where $b_k\in B,$ $c_k\in C,$ $x+y>c_k+b_k,$}\\
n(b_k)  &\text{if $(x,y)=(b_k,c_k)$ or  $(x,y)=(c_k,b_k),$ $b_k\in B,$ $c_k\in C,$} \\
\min(x,y)  &\text{if $(x,y)\in \{b_k\}\times\cint{b_k,c_k} \cup \cint{b_k,c_k} \times \{b_k\} $ and $b_k\in B, c_k\notin C,$} \\
\max(x,y)  &\text{if $(x,y)\in \{c_k\}\times\cint{b_k,c_k} \cup \cint{b_k,c_k} \times \{c_k\} $ and $b_k\notin B, c_k\in C,$}
\end{cases}$$ where $v_k=c_k$ ($v_k=b_k$) if there exists an $i\in K$ such that $b_k=a_i$ and $U_i$ is disjunctive (conjunctive)
and $v_k=n(b_k)$ if $b_k\in B,c_k\in C,$ $v_k=b_k$ if $b_k\in B,c_k\notin C,$ $v_k=c_k$ if $b_k\notin B,c_k\in C,$ and
$(U_k)^{a_k,b_k,c_k,d_k}_{v_{k}}$ is given by the formula \eqref{unitr}.
Then $U^e$ is a uninorm. \label{proorduni}
\end{proposition}

The above ordinal sum is an ordinal sum in the sense of Clifford \cite{cli},
where each semigroup is defined on $\lint{a_k,b_k} \cup \{U^e(b_k,c_k)\} \cup\rint{c_k,d_k},$ $k\in K$ and on
$\{f_l\},$ $\{g_l\},$
 where $l\in L$ and $f_l\in B,$ $g_l\in C$ for all $l\in L.$ Note that the set $B$ ($C$) is the set of accumulation points of the set
 $\{b_k\mid k\in K\}$ ($\{c_k\mid k\in K\}$) which are not contained in $\{a_k\mid k\in K\}$ ($\{d_k\mid k\in K\}$).

Although in the case of t-norms (t-conorms) the ordinal sum is simply equal to $\min$ ($\max$) on $\{b\}\times\uint \cup \{c\}\times\uint$
for $b\in B, c\in C,$ in the case of uninorms we have both possibilities $\min$  and $\max$ on the area
$\cint{0,e}\times \cint{e,1} $ and therefore the construction is a small bit more complicated here.

% such where the summand semigroups are not overlapping,
%if the predecessor of each summand
%which is a nilpotent uninorm is a uninorm with the neutral
%element $e_k$ such that $U(x,y)\neq e_k$ holds for all $(x,y)\neq (e_k,e_k)$ (see \cite{genuni}).
%Note that here nilpotent uninorm means such uninorm $U$ that $U(x,y)=u$ for some
%$x,y\neq u,$ where $u$ is the annihilator of $U,$ and if for two summands in the ordinal sum defined on $\lint{a_{k_1},b_{k_1}}\cup \rint{c_{k_1},d_{k_1}}$ and on $\lint{a_{k_2},b_{k_2}}\cup \rint{c_{k_2},d_{k_2}},$ respectively, we have  $b_{k_1}=a_{k_2}$ (i.e., also
%$c_{k_1}=d_{k_2}$) then the first is called a predecessor summand of the second.

In the case that we assume an ordinal sum of uninorms such that $\bigcup_{k\in K}\cint{a_k,b_k}\neq\cint{0,e}$
($\bigcup_{k\in K}\cint{c_k,d_k}\neq\cint{e,1}$) this can be given by the above ordinal sum, where the missing summands are covered by internal uninorms.

\begin{definition}
\begin{mylist}
\item  If for a summand $\langle a_k,b_k,c_k,d_k,U_k \rangle$ for some $k\in K$ we have
$a_k=b_k$ and $c_k=d_k$  we say that this summand is \emph{empty}.
\item   If for a summand $\langle a_k,b_k,c_k,d_k,U_k \rangle$ for some $k\in K$ we have $a_k\neq b_k$ and $c_k\neq d_k$ we say that this summand is \emph{complete}.
\item We say that an ordinal sum $U^e=(\langle a_k,b_k,c_k,d_k,U_k \rangle \mid k\in K)^e$ is  \emph{complete} when all its summands are
complete.
\item We say that a summand $\langle a_k,b_k,c_k,d_k,U_k \rangle$ for some $k\in K$ is \emph{totally employed} if one of the following conditions is satisfied:
    \begin{itemize}
    \item $U_k$ is a proper uninorm and the summand is complete,
    \item $U_k$ is a t-norm and $a_k\neq b_k,$ $c_k=d_k,$
    \item $U_k$ is a t-conorm and $a_k = b_k,$ $c_k\neq d_k.$
    \end{itemize}
\end{mylist}
\end{definition}

\begin{example}
Assume $U_1 \in \mathcal{U}_{\min}$ and $U_2\in \mathcal{U}_{\max}$ with respective neutral elements $e_1,e_2.$
Then $U_1$ and $U_2$ are ordinal sums of uninorms,
$U_1=(\langle e_1,e_1,e_1,1,C_{U_1}\rangle,\langle 0,e_1,1,1,T_{U_1}\rangle)^{e_1}$  and $U_2=(\langle 0,e_2,e_2,e_2,T_{U_2}\rangle,\langle 0,0,e_2,1,C_{U_2}\rangle)^{e_2}.$ If we denote $U_k^* = (U_k)^{a_k,b_k,c_k,d_k}_{v_{k}}$ for all $k\in K$ in the respective ordinal sum we can see uninorms $U_1$ and $U_2$ on Figure \ref{figst1}.  \label{exmm}
\end{example}

We will continue by discussing several peculiarities that appears in the ordinal sum construction.

\section{Notes on ordinal sum of uninorms}

Our aim in this paper is to express each uninorm, such that $T_U$ and $C_U$ are continuous, as ordinal sum of Archimedean uninorms, i.e., such where underlying t-norm and t-conorm are Archimedean.
  First we recall two results from
\cite{repord} (compare also \cite{FB12}).

\begin{proposition}
Assume a uninorm $U\colon \uint^2\lran \uint,$ $U\in \mathcal{U}\cap \mathcal{N}.$
If $T_U$ and $C_U$ are Archimedean then either $U$  is a representable uninorm or $U\in \mathcal{N}_{\min}\cup \mathcal{N}_{\max}$.  \label{sumstrp}
\end{proposition}

\begin{proposition}
Assume a uninorm $U\colon \uint^2\lran \uint$ such that $U\in \mathcal{U}$ and $U\notin \mathcal{N}.$ Then $U$ is an ordinal sum of a uninorm and a non-proper uninorm
(i.e., a t-norm or a t-conorm). \label{pron}
\end{proposition}

The previous two propositions imply that if for a uninorm $U\in \mathcal{U}$ both underlying operations are Archimedean then either  $U\in \mathcal{N}_{\min}\cup \mathcal{N}_{\max}$ or
 $U\in \mathcal{U}_{\min}\cup \mathcal{U}_{\max}$ or $U$ is a representable uninorm.

In the following we will say that a uninorm is irreducible with respect to the ordinal sum construction if it can be expressed only as a trivial ordinal sum, i.e., with the only non-empty summand  defined on $(\lint{0,e}\cup\rint{e,1})^2,$ or on $\uint^2.$

\begin{proposition}
Let $U\colon \uint^2\lran \uint$ be a  uninorm such that $U\in \mathcal{U}$ and $U\in \mathcal{N}_{\min}$  ($U\in  \mathcal{N}_{\max}$). Then if $U$ has no idempotent points in $\opint{e,1}$ ($\opint{0,e}$).
$U$ is irreducible with respect to the ordinal sum construction.
\end{proposition}

\begin{proof}
Assume $U\in \mathcal{U}$ and $U\in \mathcal{N}_{\min} $ (the case when  $U\in \mathcal{N}_{\max} $ can be shown analogically). If $U$ would be an ordinal sum of at least two uninorms, then there would exist $a,b\in \uint,$ $a\leq e\leq b$ such that $U$ is an ordinal sum of two uninorms  one acting on  $(\lint{0,a}\cup \rint{b,1})^2$
and the other on $\cint{a,b}^2.$ Then  $U(x,y)=\max(x,y)$  for $(x,y)\in \cint{a,b}\times\rint{b,1},$ however, since
$U(x,y)=\min(x,y)$ for all $(x,y)\in \lint{0,e}\times\lint{e,1},$ we get either $a=e$ or $b=1.$ If $b=1$ then
the ordinal sum construction implies $U(x,y)=\min(x,y)$ on $ \lint{0,a}\times\cint{a,b},$ i.e., $U(x,1)=x$ for $x<a$ what means that either $a=0$ or
$U\notin\mathcal{N}.$ Since $U\in \mathcal{N}$ assume $a=0.$ However, in such a case we have summands on $(\lint{0,0}\cup \rint{1,1})^2$ and on $\uint^2,$ i.e., the ordinal sum construction is trivial.  Therefore suppose that $a=e,$ $b<1.$ If $b\neq e$ then $b$ is a non-trivial idempotent point in $\opint{e,1},$ i.e., we get
$b=e$ and thus the ordinal sum construction is again trivial.
\end{proof}

The previous proposition shows that there is a problem with uninorms $U\in \mathcal{N}$ such that there exists a uninorm $U_1,$ which is an ordinal sum of a non-proper uninorm and a uninorm, and $U=U_1$  on $\opint{0,1}^2.$

Further, we have the following result.

\begin{lemma}
Let $U\colon \uint^2\lran \uint$ be a  uninorm. Then both binary functions $U^*\colon \uint^2\lran \uint$ and
$U_*\colon \uint^2\lran \uint$ given by
\begin{equation}U^*(x,y)=\begin{cases} 1 &\text{if $\max(x,y)=1,$}\\
0 &\text{if $\min(x,y)=0,\max(x,y)<1,$}\\
U(x,y) &\text{otherwise}\\\end{cases}\label{eqbor1}\end{equation} and
\begin{equation}U_*(x,y)=\begin{cases} 0 &\text{if $\min(x,y)=0,$}\\
1 &\text{if $\min(x,y)>0,\max(x,y)=1,$}\\
U(x,y) &\text{otherwise}\\\end{cases}\label{eqbor2}\end{equation} are  uninorms if and only if  there are no such $x_1,x_2\in \opint{0,1}$ that
$U(x_1,x_2)\in \{0,1\}.$   \label{lem2}
\end{lemma}

\begin{proof}
If there are no  $x_1,x_2\in \opint{0,1}$ such that
$U(x_1,x_2)\in \{0,1\}$ then  $U$ can be restricted to $\opint{0,1}^2$  and then  evidently  the functions $U^*$ and $U_*$ are uninorms.
\end{proof}

\begin{lemma}
Let $U\colon \uint^2\lran \uint$ be a  uninorm, $U\in \mathcal{U},$ such that there exist $x_1,x_2\in \opint{0,1}$ such that $U(x_1,x_2)\in \{0,1\}.$  Then $U^*$ from the previous lemma is a uninorm if and only if $U=U_1$  on $\opint{0,1}^2,$ where $U_1 \in  \mathcal{U}$ is an ordinal sum of a t-norm and a uninorm. Further,
$U_*$ from the previous lemma  is a uninorm if and only if $U=U_2$  on $\opint{0,1}^2,$ where $U_2 \in  \mathcal{U}$ is an ordinal sum of a t-conorm and a uninorm. \label{lem2a}
\end{lemma}

\begin{proof}
Assume $U(x_1,x_2)=0$  for some $x_1,x_2\in \opint{0,1}$ (similar result can be obtained in the case when $U(x_1,x_2)=1$). Then if $U_*$ is associative we have $1=U_*(x_1,1)=U_*(x_1,x_2,1)=U_*(0,1)=0$ what is a contradiction. Thus $U_*$ is not a uninorm. Further, evidently $x_1<e$ and $x_2<e$ and $T(x,x)<x$ for all $x\in \opint{0,f},$ where $f$ is the smallest idempotent point of $U$ such that $f\geq \max(x_1,x_2).$
Since each continuous  t-norm can be expressed as an ordinal sum of continuous Archimedean t-norms, and since for each nilpotent t-norm $T_N$ there exists a continuous strictly decreasing function $r_T\colon \uint\lran \uint$ such that  $T(x,y)=0$ if and only if $y\leq r_T(x)$ we see that there exists a continuous strictly decreasing function $r^*\colon \cint{0,f}\lran \cint{0,f}$ such that $U(x,y)=0$ for $(x,y)\in \rint{0,1}^2$ if and only if $y\leq r^*(x).$ Assume $x\in \opint{0,f}.$
Then for $U^*$ and $y\in \lint{e,1}$ we have $U^*(x,r^*(x),y)=U^*(0,y)=0$ and $U^*(x,y)\geq x,$  $U^*(r^*(x),y)\geq r^*(x).$
If $U^*(x,y)> x,$ or if   $U^*(r^*(x),y)> r^*(x)$ then $U^*(x,r^*(x),y)>0$
what is a contradiction, i.e.,  $U^*(x,y)= x$ for all $x\in \opint{0,f}.$
 Thus up to the border of the unit square $U^*$ is an ordinal sum of a t-norm on $\cint{0,f}^2$ and a uninorm on $\cint{f,1}^2.$
\end{proof}

\begin{example}
Let $U\colon \uint^2\lran \uint$ be a  uninorm $U\in \mathcal{U}$ and $U\in \mathcal{N}_{\min}, $ such that $T_U$ and $C_U$ are Archimedean.
Then $U$ is irreducible with respect to the ordinal sum construction for uninorms. However, $U$ can be still decomposed into an ordinal sum of semigroups in the sense of Clifford \cite{cli}. Then evidently summands supports should be unions of the following sets
$\{0\},\opint{0,e},\{e\},\opint{e,1},\{1\}.$ First assume $U(0,1)=1.$ Then similarly as in the proof of the previous lemma $C_U$ cannot be nilpotent, i.e.,  $U(x,y)=1$ if and only if $\max(x,y)=1.$ Thus $U$ is an ordinal sum of semigroups such that for the ordering of the semigroups $U_I$ respective to the support $I$ we have
$U_{\{1\}}<  U_{\lint{0,e}}< U_{\opint{e,1}} < U_{\{e\}}.$  Now assume $U(0,1)=0.$ Then again $T_U$ cannot be nilpotent and if $C_U$ is nilpotent from the previous lemma we see that $U(x,y)=\max(x,y)$ for some $(x,y)\in \opint{0,e}\times\opint{e,1},$ what is a contradiction. Thus both $T_U$ and $C_U$ are strict. Then $U$ can be obtained as an ordinal sum, where the order of semigroups is $U_{\{0\}} < U_{\{1\}} < U_{\opint{0,e}}< U_{\opint{e,1}} < U_{\{e\}}.$

\end{example}

If $U$ is such that $U(x,y)\in \opint{0,1}$ for all $(x,y)\in \opint{0,1}^2$ then there is no problem to decompose $U$ into three semigroups,
defined respectively  on $\{0\},$ $\opint{0,1}$ and $\{1\}.$ Thus we have to investigate the situation when underlying t-norm or t-conorm is nilpotent.

\begin{lemma}
Let $U\colon \uint^2\lran \uint$ be a  uninorm, $U\in \mathcal{U},$ and let $U(x_1,x_2)=0$ ($U(x_1,x_2)=1$) for some $x_1,x_2\in \opint{0,1}.$ Then $U$ is a non-trivial ordinal sum in the sense of Clifford.
\end{lemma}

\begin{proof} We will show only the case when $U(x_1,x_2)=0$ (the case when $U(x_1,x_2)=1$ is analogous).
If $U\in \mathcal{N}$ then Lemma \ref{lem2} implies that $U(0,1)=1,$ and there is no $y_1,y_2\in \opint{0,1}$  such that $U(y_1,y_2)=1$ and  $U=U_1$  on $\opint{0,1}^2,$ where $U_1 \in  \mathcal{U}$ is an ordinal sum of a t-norm and a uninorm. Thus $U$ is an ordinal sum in the sense of Clifford where the order of semigroups is $U_{\{1\}}  < U_{\lint{0,a}}< U_{\lint{a,1}} $ for some corresponding $a\in \rint{0,e}.$

Suppose $U\notin \mathcal{N}.$ Then Proposition \ref{pron} implies that
 $U$ is an ordinal sum of a non-proper uninorm and a uninorm, i.e., either the order of semigroups is  $U_{\lint{0,a}}< U_{\cint{a,1}} $ or
 $U_{\cint{0,a}}> U_{\rint{a,1}} $ for some corresponding $a\in \opint{0,1}.$
\end{proof}

If we summarise these results, although some uninorms  $U\in \mathcal{U}$ cannot be decomposed into ordinal sum of uninorms with Archimedean underlying t-norm and t-conorm it seems that each uninorm  $U\in \mathcal{U}$ can be expressed as an ordinal sum of semigroups, in the sense of Clifford, such that each summand semigroup is either internal, or representable uninorm, or continuous Archimedean t-norm or continuous Archimedean t-conorm. We will try to prove this fact in the next section.

In the following we will define an extended ordinal sum construction. First we recall a result from
\cite[Proposition 2]{MMT}.

\begin{proposition}
An internal, commutative, non-decreasing binary function $O\colon \uint^2\lran \uint $ is associative. \label{coroass1}
\end{proposition}

\begin{corollary}
Let $O\colon \uint^2\lran \uint $ be a commutative, non-decreasing binary function. If $O(x,y)\in \{x,y\},$ $O(x,z)\in \{x,z\}$
and $O(y,z)\in \{y,z\}$ for some $x,y,z\in \uint$  then $O(O(x,y),z)=O((y,z),x)=O(O(x,z),y).$ \label{lemass}
\end{corollary}

The ordinal sum construction implies that if $a\in \uint$ is and idempotent element of the ordinal sum uninorm $U$ then
$U(a,x)\in \{a,x\}$ for all $x\in \uint.$ We recall also a result from \cite{mulfun}.

\begin{lemma}
Let  $U\colon \uint^2\lran \uint$ be a uninorm and let $U\in \mathcal{U}.$ If $a\in \uint$ is an idempotent point of $U$ then
$U$ is internal on $\{a\}\times \uint.$
\end{lemma}

Now we can define an extended ordinal sum of uninorms. Although the construction seems to be quite complicated, since all sets have to be properly defined, in fact it is quite simple. We will see that extended ordinal sum  and  ordinal sum  can differ only in the points from $\{a\}\times \uint$
($ \uint \times \{a\}$), for all idempotent elements $a$  of the corresponding  (extended) ordinal sum $U.$ Since in both cases
$U(a,x)\in \{a,x\}$ for all $x\in \uint,$ the difference between these two constructions is determined the point where on $\{a\}\times \uint$ the $\min$
will change to $\max.$ 
If we focus on points from $\{a_k,b_k,c_k,d_k\}_{k\in K}$ as possible points of change, due to the monotonicity  the difference between these two constructions  can occur only in the case when $a_k=b_k=a$ or $c_k=d_k=a$ for some
$k\in K.$ Further, from Lemmas \ref{lem2} and \ref{lem2a} we see that
the possible points for this change are influenced also by the existence of nilpotent elements, i.e., such $x_1,x_2\in \opint{0,1}$ where
$U(x_1,x_2)\in \{0,1\}.$

First we define a c-strict (composite strict) and c-nilpotent  t-norm and t-conorm.

\begin{definition}
A continuous t-norm   $T\colon \uint^2\lran \uint$ (t-conorm   $C\colon \uint^2\lran \uint$) will be called
\emph{c-strict} if $T(x,y)\in \opint{0,1}$ ($C(x,y)\in \opint{0,1}$ ) for all $(x,y)\in \opint{0,1}^2.$ In the other case $T$ ($C$) will be called \emph{c-nilpotent}.
\end{definition}

\begin{proposition}
Let $U^e\colon \uint^2\lran \uint$  be a  uninorm such that $U^e=(\langle a_k,b_k,c_k,d_k,U_k \rangle \mid k\in K)^e,$ where all conditions of Proposition \ref{proorduni} are satisfied. Denote $G=\{b_k \mid k\in K, a_k=b_k\neq e , U(b_k,c_k)=b_k\},$ $H=\{c_k \mid k\in K, c_k=d_k\neq e, U(b_k,c_k)=c_k \},$
and for   $x\in G$ denote
$G_x=\{k\in K \mid b_k=x\},$ for $x\in H$ denote $H_x=\{k\in K \mid c_k=x\}.$
Let $G_x^{**}$ be the closure of the set $\{c_k\mid k\in G_x\}$
and denote $G_x^*=G_x^{**} \setminus \{d_i\}_{i\in K},$ and  let $H_x^{**}$ be the closure of the set $\{b_k\mid k\in H_x\}$
and denote $H_x^*=H_x^{**} \setminus \{a_i\}_{i\in K}.$
 Further, for $k\in G_x,$ $x\in G$ denote
$$F_k=\begin{cases}\{ \{c_k\},\lint{c_k,d_k},\cint{c_k,d_k}\} &\text{if $C_{U_k}$ is c-strict,} \\
\{ \{c_k\},\cint{c_k,d_k}\} &\text{if $C_{U_k}$ is c-nilpotent,} \end{cases}$$
for $c\in G_x^*$ denote
$$F^*_c=\begin{cases} \{\cint{0,c}\} &\text{if $c=\inf\{c_k\mid k\in G_x\},$} \\\{\lint{0,c},\cint{0,c}\} &\text{else,}\end{cases}$$
and for  $k\in H_x,$ $x\in H$ denote
$$J_k=\begin{cases}\{\emptyset, \{a_k\},\lint{a_k,b_k}\} &\text{if $T_{U_k}$ is c-strict,} \\
\{\emptyset, \lint{a_k,b_k}\} &\text{if $T_{U_k}$ is c-nilpotent,} \end{cases}$$
and for $c\in H_x^*$ denote
$$J^*_b=\begin{cases} \{\lint{0,b}\} &\text{if $b=\sup\{b_k\mid k\in H_x\},$} \\ \{\lint{0,b},\cint{0,b}\} &\text{else.}\end{cases} $$

With the convention $S\cup \{S_1,S_2\}=\{S\cup S_1,S\cup S_2\}$  let
$g\colon G\lran \bigcup\limits_{x\in G}(\bigcup\limits_{k\in G_x }(\lint{0,c_k}\cup F_k) \cup \bigcup\limits_{c\in G_x^*}F^*_c)$ be a function such that
$g(x)\in \bigcup\limits_{k\in G_x }(\lint{0,c_k}\cup F_k)\cup \bigcup\limits_{c\in G_x^*}F^*_c$  and let $h\colon H\lran \bigcup\limits_{x\in H}(\bigcup\limits_{k\in H_x }(\lint{0,a_k}\cup J_k)\cup \bigcup\limits_{b\in H_x^*}J^*_b)$ be a function such that
$h(x)\in \bigcup\limits_{k\in H_x }(\lint{0,a_k}\cup J_k)\cup \bigcup\limits_{b\in H_x^*}J^*_b.$

Then the binary function $V^e\colon \uint^2\lran \uint$ given by
$$V^e(x,y)=\begin{cases} U^e(x,y) &\text{ if $(x,y)\in (\uint\setminus (G\cup H))^2,$} \\
\min(x,y) &\text{ if $x\in G,$ $y\in g(x),$ or $y\in G,$ $x\in g(y),$ } \\
\max(x,y) &\text{ if $x\in G,$ $y\notin g(x),$ or $y\in G,$ $x\notin g(y),$} \\
\min(x,y) &\text{ if $x\in H,$ $y\in h(x),$ or $y\in H,$ $x\in h(y),$} \\
\max(x,y) &\text{ if $x\in H,$ $y\notin h(x),$ or $y\in H,$ $x\notin h(y)$} \\
 \end{cases}$$ is a uninorm, which will be called an extended ordinal sum of uninorms.
   We write $V^e=(\langle a_k,b_k,c_k,d_k,U_k \rangle \mid k\in K)^e.$
      Further,  $V^e\in \mathcal{U}$ if and only if $U^e\in \mathcal{U}.$
    \label{proext}
\end{proposition}
The proof of this result can be found in Appendix.

It is evident that ordinal sum of uninorms is a special case of extended ordinal sum of uninorms.

\begin{example}
Assume an  ordinal sum $U^e=(\langle 0,e,e,e,T \rangle , \langle 0,0,e,b,C_1 \rangle, \langle 0,0,b,1,C_2 \rangle )^e,$ for some
$b,e \in \uint,$ $0<e<b<1$ (see Figure \ref{figexo}) and a t-norm $T$ and t-conorms $C_1,$ $C_2.$
Assume that $C_1$ is c-strict and $C_2$ is c-nilpotent.
 Since $T$ is a t-norm we have
$U^e(0,e)=0$ and thus we have $G=\{0\},$ $H=\emptyset$ for  sets $G,H$ from the previous proposition.  If we denote the summands respectively as
$1,2,3$ for $K=\{1,2,3\}$ we get $G_0=\{2,3\}.$ Further, $G_x^{**}=\{e,b\},$ $G_x^*=\emptyset,$
$$F_2 = \{\{e\},\lint{e,b},\cint{e,b}\},$$ $$F_3 = \{\{b\},\cint{b,1}\}.$$ The function $g$ is defined only in one point $0$ and its range is the set $$\{\cint{0,e},\lint{0,b},\cint{0,b},\cint{0,1}\}.$$
Thus if $V^e=(\langle 0,e,e,e,T \rangle , \langle 0,0,e,b,C_1 \rangle, \langle 0,0,b,1,C_2 \rangle )^e,$ is an extended ordinal sum
we have $V^e=U^e$ if $g(0) = \cint{0,e}.$ Further we can define three other different extended ordinal sums by respectively selecting a different value/interval for $g(0).$ It is evident that $V^e$ and $U^e$ may differ only on $\{0\}\times \uint \cup \uint \times \{0\}.$
Sketch of a more complicated example can be seen on Figure \ref{figcomx}.
\label{exexo}
\end{example}

\begin{figure}[htb] \begin{center}
\begin{picture}(150,150)
\put(0,0){ \line(1,0){150} }
\put(0,0){ \line(0,1){150} }
\put(150,150){ \line(-1,0){150} }
\put(150,150){ \line(0,-1){150} }

\put(0,50){ \line(1,0){150} }
\put(50,0){ \line(0,1){150} }
\put(0,100){ \line(1,0){150} }
\put(100,0){ \line(0,1){150} }

\put(70,70){$C_1^*$}
\put(20,20){$T^*$}
\put(120,120){$C_2^*$}
\put(70,120){$\max$}
\put(120,70){$\max$}
\put(70,20){$\max$}
\put(20,70){$\max$}
\put(120,20){$\max$}
\put(20,120){$\max$}

\end{picture} \end{center}
\caption{The ordinal sum uninorm $U^{e}$ from Example \ref{exexo}.} \label{figexo}
\end{figure}

\begin{figure}[htb] \begin{center}
\begin{picture}(200,200)
\put(0,0){ \line(1,0){200} }
\put(0,0){ \line(0,1){200} }
\put(200,200){ \line(-1,0){200} }
\put(200,200){ \line(0,-1){200} }
\put(5,8){{\tiny $m+1$}}
\put(5,190){{\tiny $m+1$}}
\put(185,190){{\tiny $m+1$}}
\put(185,8){{\tiny $m+1$}}
\put(0,20){ \line(1,0){200} }
\put(20,0){ \line(0,1){200} }
\put(200,180){ \line(-1,0){200} }
\put(180,200){ \line(0,-1){200} }
\put(100,100){ \line(-1,0){80} }
\put(100,100){ \line(0,-1){80} }
\put(55,55){{ $1$}}
\put(100,100){ \line(1,0){20} }
\put(100,100){ \line(0,1){20} }
\put(120,120){ \line(-1,0){20} }
\put(120,120){ \line(0,-1){20} }
\put(107,107){{ $2$}}
\put(120,120){ \line(1,0){10} }
\put(120,120){ \line(0,1){10} }
\put(130,130){ \line(-1,0){10} }
\put(130,130){ \line(0,-1){10} }
\put(123,121){{ $3$}}
\put(130,130){ \line(1,0){20} }
\put(130,130){ \line(0,1){20} }
\put(150,150){ \line(-1,0){20} }
\put(150,150){ \line(0,-1){20} }
\put(138,136){{ $4$}}
\put(166,166){{ $m$}}
\put(156,152){\circle*{2}}
\put(159,155){\circle*{2}}
\put(162,158){\circle*{2}}
\put(160,160){ \line(1,0){20} }
\put(160,160){ \line(0,1){20} }
\put(180,180){ \line(-1,0){20} }
\put(180,180){ \line(0,-1){20} }
\put(184,100){$\max$}
\put(5,90){$\min$}
\put(100,184){$\max$}
\put(90,4){$\min$}
\put(60,140){$\max$}
\put(140,60){$\max$}
\put(25,140){\oval(20,80)}
\put(144,20){\oval(80,20)}

\end{picture}
\end{center}
\caption{Sketch of a uninorm which is an  ordinal sum with $m+1$ summands. The summands $1$ and $m+1$ are complete, the others are non-complete. The rounded area (the line in the center) designated the place where the ordinal sum construction and the extended ordinal sum construction can differ. }
\label{figcomx}
\end{figure}
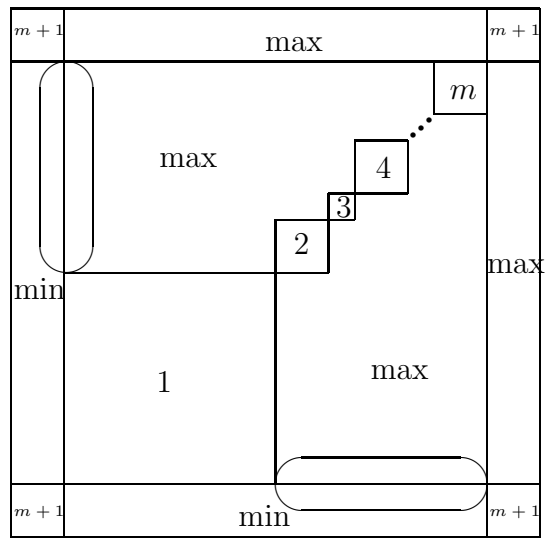

\begin{remark}
As we mentioned above, if  $U^e=(\langle a_k,b_k,c_k,d_k,U_k \rangle \mid k\in K)^e$ is an ordinal sum  and $V^e=(\langle a_k,b_k,c_k,d_k,U_k \rangle \mid k\in K)^e$ is an extended ordinal sum and $U^e$ differs from $V^e$ in the point $(x,y)$ then
$U^e(x,y)\in \{x,y\}$ and $U^e(x,y) +V^e(x,y) = x+y.$ Since $V^e$  and $U^e$ are non-decreasing then if $(x,y)\in \opint{0,1}^2$  necessarily 
the point $(x,y)$ is a point of discontinuity of both $U^e$ and $V^e.$ Thus $(x,y)$ belongs to the graph of the characterizing multi-functions of
both $V^e$  and $U^e.$ Similar result can be shown for $(x,y) \in \uint^2,$ $(x,y)\notin \opint{0,1}^2.$
\end{remark}

Since ordinal sum of t-norms (t-conorms) is continuous if and only if all its summands are continuous we easily get the following.

\begin{proposition}
Let $V^e=(\langle a_k,b_k,c_k,d_k,U_k \rangle \mid k\in K)^e$ be an extended ordinal sum of uninorms such that for all $k\in K$
we have $U_k \in \mathcal{U}.$ Then $V^e\in \mathcal{U}.$
\end{proposition}

For the opposite claim we should exclude the cases when, for example, only the underlying t-norm is continuous, however, if the summand is not complete and only the t-norm part is employed still the resulting extended ordinal sum can belong to $\mathcal{U}.$

\begin{proposition}
Let $V^e=(\langle a_k,b_k,c_k,d_k,U_k \rangle \mid k\in K)^e$ be an extended ordinal sum of uninorms such that all summands are totally employed.
Then if $V^e\in \mathcal{U}$ also  $U_k \in \mathcal{U}$ for all $k\in K.$
\end{proposition}

\section{Uninorms  with continuous underlying functions}

In this section we will examine a decomposition of a uninorm $U\in \mathcal{U}$ into representable uninorms with respect to the extended ordinal sum construction. First we recall some useful results from \cite{mulfun}.

\begin{definition}
A mapping $p\colon X\lran \mathcal{P}(Y)$ is called a multi-function if to every $x\in X$ it assigns a subset of $Y,$ i.e.,
$p(x)\subseteq Y.$ A multi-function $p$ is called

\begin{mylist}
\item \emph{non-increasing} if  for all $x_1,x_2\in X,$ $x_1<x_2$ there is $p(x_1)\geq p(x_2),$ i.e, for all $y_1\in p(x_1)$ and all $y_2\in p(x_2)$ we have
$y_1\geq y_2$ and thus $\mathrm{Card}(p(x_1)\cap p(x_2))\leq 1,$
\item  \emph{symmetric} if $y\in p(x)$ if and only if $x\in p(y).$
\end{mylist}
 The graph of a multi-function $p$ will be denoted by $G(p),$ i.e., $(x,y)\in G(p)$ if and only if
$y\in p(x).$
\end{definition}

\begin{lemma}
A symmetric multi-function $p\colon \uint\lran \uint$ is surjective, i.e.,  for all $y\in Y$ there exists an $x\in X$ such that $y\in p(x),$ if and only if we have $p(x)\neq \emptyset$ for all $x\in X.$
 The graph of a symmetric, surjective,  non-increasing  multi-function $p\colon \uint \lran \uint$ is a connected line. \label{lemnone}
\end{lemma}

We will denote the set of all uninorms  $U\colon \uint^2\lran \uint$ such that  $U$ is continuous on
$\uint^2\setminus R,$ where $R=G(r)$ and $r$ is a symmetric,  surjective, non-increasing multi-function such that  $U(x,y)=e$ implies $(x,y)\in R,$
by $\mathcal{UR}.$

\begin{theorem}[\cite{mulfun}]
Let $U\colon \uint^2\lran \uint$ be a uninorm. Then $U\in \mathcal{U}$ if and only if $U\in \mathcal{UR}$ and  in each point $(x,y)\in \uint^2$ the uninorm $U$ is either left-continuous or right-continuous.
\end{theorem}

\begin{remark}\begin{mylist}
\item
Note that although for $U\in \mathcal{U}$ the previous theorem implies that $U$ is continuous on $\uint^2\setminus R$ for $R=G(r)$ it does not mean that all points of $R$ are points of discontinuity of $U.$ In fact, in \cite{mulfun} it was shown that
for a uninorm $U\in \mathcal{U}$ either $U(x,y)=e$ implies $x=y=e$ or there exists a non-empty interval $\opint{a,d}$ such that
$U(x,y)=e$ if and only if $x,y\in \opint{a,d}.$ In the later case $U$ is continuous in all points from  $\uint^2\setminus (\cint{0,a}\cup \cint{d,1})^2.$
Moreover, for a conjunctive uninorm $U\in \mathcal{U}$ (similarly for a disjunctive uninorm $U\in \mathcal{U}$) we have either
$U(x,1)=1$ for all $x>0$ or $U(x,1)<e$ for some $0<x<e.$ In the later case  $U$ is continuous in all points from $\lint{0,x} \times \uint \cup \uint\times \lint{0,x}.$
\item  The graph of a  symmetric,  surjective, non-increasing  multi-function can be divided into  connected maximal segments which are either strictly decreasing, or the horizontal segments, or the vertical segments. Note that a horizontal segment corresponds to a closed interval $Z$ such that there exists a $y\in \uint$ with $ r(x)=\{y\}$ for all $x\in\mathrm{int}(Z)$ (where $\mathrm{int}(Z)$ is the interval $Z$ without border points) and a vertical segment corresponds to a closed interval $V$ such that $ r(x)=V,$ $\mathrm{Card}(V)>1,$ for some $x\in \uint.$  We say that segment corresponding to an interval $S$ is strictly decreasing if  $y_1 \in r(x_1),$ $y_2\in r(x_2)$ for
$x_1,x_2\in S,$ $x_1<x_2$ implies $y_2<y_1$ and $\mathrm{Card}(r(x))=1$ for all $x\in \mathrm{int}(S).$ The previous description implies that all horizontal, vertical and strictly decreasing  segments correspond to closed intervals.
\end{mylist}  \label{remseg}
\end{remark}

We recall an example from \cite{mulfun}.

\begin{example}
Assume a representable uninorm  $U_1\colon \uint^2\lran \uint$  and a continuous t-norm $T\colon \uint^2\lran \uint$ and a continuous t-conorm
$C\colon \uint^2\lran \uint.$ For $e=\frac{1}{2}$ their ordinal sum $U^{\frac{1}{2}}=(\langle \frac{1}{4},\frac{1}{2},\frac{1}{2},\frac{3}{4},U_1 \rangle, \langle 0,\frac{1}{4},\frac{3}{4},\frac{3}{4},T \rangle, \langle 0,0,\frac{3}{4},1,C \rangle )^{\frac{1}{2}}$ is a uninorm, $U^{\frac{1}{2}} \in \mathcal{U}.$ For simplicity we will assume that $\frac{1}{2}$ is the neutral element of $U_1$ and that $U_1(x,1-x)=\frac{1}{2}$ for all
$x\in \opint{0,1}.$
 On Figure \ref{figcom} we can see the characterizing multi-function $r$ of $U^{\frac{1}{2}}$ as well as its set of discontinuity points.
\label{exla}
\end{example}

\begin{figure}[htb] \hskip2cm
\begin{picture}(150,150)
\put(0,0){ \line(1,0){150} }
\put(0,0){ \line(0,1){150} }
\put(150,150){ \line(-1,0){150} }
\put(150,150){ \line(0,-1){150} }

\put(0,50){ \line(1,0){100} }
\put(50,0){ \line(0,1){100} }
\put(0,100){ \line(1,0){150} }
\put(100,0){ \line(0,1){150} }

\put(70,70){$U_1^*$}
\put(25,20){$T^*$}
\put(125,120){$C^*$}
\put(50,120){$\max$}
\put(120,50){$\max$}
\put(70,20){$\min$}
\put(20,70){$\min$}

\linethickness{0.5mm}
\put(0,100){ \line(1,0){50} }

\put(100,0){ \line(0,1){50} }
\end{picture}
 \hskip2cm
\begin{picture}(150,150)
\put(0,0){ \line(1,0){150} }
\put(0,0){ \line(0,1){150} }
\put(150,150){ \line(-1,0){150} }
\put(150,150){ \line(0,-1){150} }

\put(0,50){ \line(1,0){100} }
\put(50,0){ \line(0,1){100} }
\put(0,100){ \line(1,0){150} }
\put(100,0){ \line(0,1){150} }

\put(50,100){ \line(1,-1){50} }

\put(62,62){$U_1^*$}
\put(82,82){$U_1^*$}
\put(25,20){$T^*$}
\put(125,120){$C^*$}
\put(50,120){$\max$}
\put(120,50){$\max$}
\put(70,20){$\min$}
\put(20,70){$\min$}

\linethickness{0.5mm}
\put(0,100){ \line(1,0){50} }
\put(0,100){ \line(0,1){50} }
\put(100,0){ \line(1,0){50} }
\put(100,0){ \line(0,1){50} }
\end{picture}
\caption{The uninorm $U^{\frac{1}{2}}$ from Example \ref{exla}. Left: the bold lines denote the points of discontinuity of $U^{\frac{1}{2}}.$
Right: the oblique and bold lines denotes the characterizing multi-function of $U^{\frac{1}{2}}.$}
\label{figcom}
\end{figure}
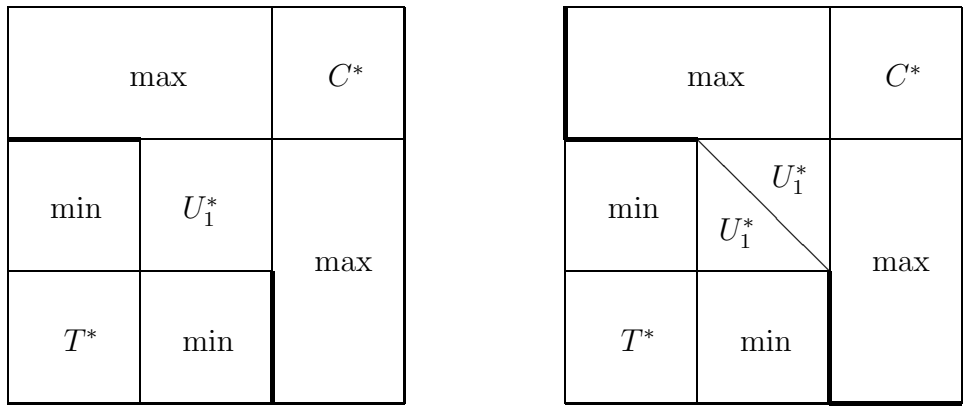
We  will now continue to examine a decomposition of a uninorm $U\in \mathcal{U}.$

\begin{lemma}
Let $U\colon \uint^2\lran \uint$ be a  uninorm $U\in \mathcal{U}$  and let $c\in \uint$ be an idempotent element.
Then if $a$ and $b$ are idempotent elements such that $U(x,x)\neq x$ for all $x\in \opint{a,b}$ we have either $U(c,x)=\min(x,c)$
for all $x\in \opint{a,b}$ or  $U(c,x)=\max(x,c)$
for all $x\in \opint{a,b}.$ Moreover, if $U$ is on $\cint{a,b}^2$ a nilpotent t-norm (nilpotent t-conorm) then
we have either $U(c,x)=\min(x,c)$
for all $x\in \lint{a,b}$ ($x\in \rint{a,b}$) or  $U(c,x)=\max(x,c)$
for all $x\in \lint{a,b}$ ($x\in \rint{a,b}$). \label{lemful}
\end{lemma}

\begin{proof} If $c=e$ the claim evidently holds. Assume $c\neq e.$
If there exists $x_1,x_2 \in \opint{a,b}$ such that $U(c,x_1)=\min(x_1,c)$ and $U(c,x_2)=\max(x_2,c)$ then it is evident that either $x_1< x_2\leq e$ and $c>e$ or $e\leq x_1< x_2$ and $c<e.$ We will assume $x_1<x_2\leq e$ and $c>e$ as the other case is analogous.
Now since $T_U$ is continuous and $U(x,x)\neq x$ for all $x\in \opint{a,b}$ there exists an $n\in \mathbb{N}$ such that $U(\underbrace{x_2,\ldots,x_2}_{n\text{-times}})<x_1.$ Moreover, $$c=U(c,x_2)=U(c,\underbrace{x_2,\ldots,x_2}_{n\text{-times}})\leq U(c,x_1)=x_1$$ what is a contradiction. If $U$ is a nilpotent t-norm on $\cint{a,b}^2$ then $U(x_1,x_2)=a$ for some $x_1,x_2\in \opint{a,b}$
and $U(a,c)=U(x_1,x_2,c)$ and the result follows.
\end{proof}

Next we include the result of \cite[Theorem 5.1]{ordut}.
Here $\mathcal{U}(e)=\{U\colon \uint^2 \lran \uint \mid U \text{ is associative, non-decreasing, with the neutral element } e\in \uint\}.$ Thus $U\in \mathcal{U}(e)$ is a uninorm if it is commutative.

\begin{theorem}
Let $U\in \mathcal{U}(e)$ and $a,b,c,d \in \uint,$ $a\leq b\leq e\leq  c\leq d$ be such that $U\vert_{\cint{a,b}^2}$ is associative, non-decreasing, with the neutral element $b$ and
$U\vert_{\cint{c,d}^2}$ is associative, non-decreasing, with the neutral element $c.$ Then the set $(\cint{a,b}\cup \cint{c,d})^2$ is closed under $U.$ \label{th51}
\end{theorem}

\begin{proposition}
Let $U\colon \uint^2\lran \uint$ be a  uninorm. If $a,b\in \cint{0,e}$ and $c,d\in \cint{e,1}$ are idempotent elements, $a\leq b$ and $c\leq d,$ then $(\lint{a,b} \cup \{U(b,c)\}\cup \rint{c,d})^2$ is closed under $U.$
\label{restr}
\end{proposition}

\begin{proof}
From the previous theorem we know that the set $(\cint{a,b}\cup \cint{c,d})^2$ is closed under $U.$ Since $b$ and $c$ are idempotent points we have
$U(b,c)\in \{b,c\}.$ If $b=c=e$ then the claim evidently holds. Suppose $b\neq c$ and
 $U(b,c)=b$ (the case when $U(b,c)=c$ is analogous). Assume that there are $x,y\in (\cint{a,b} \cup \rint{c,d})^2$
such that $U(x,y)=c.$ 
If $x\in \cint{a,b}$ (similarly if $y\in \cint{a,b}$) then $$b=U(c,b)=U(y,x,b)=U(y,x)=c$$ what is a contradiction.
Thus both $x,y\in \rint{c,d}.$ Then, however, $$c=U(x,y)\geq \max(x,y)>c $$ what is again a contradiction.
Thus $(\lint{a,b} \cup \{U(b,c)\}\cup \rint{c,d})^2$ is closed under $U.$
\end{proof}

Before we continue we recall a result from \cite{mulfun}.

\begin{proposition}
Let $U\colon \uint^2\lran \uint$ be a uninorm, $U\in \mathcal{U}.$ Then for each $x\in \uint$ there is at most one point of discontinuity of $u_x.$
Further, if $u_x$ is non-continuous in $y\in \uint$ then $U(x,z)<e$ for all $z<y$ and $U(x,z)>e$ for all $z>y.$ \label{atmost}
\end{proposition}

\begin{proposition}
Let $U\colon \uint^2\lran \uint$ be a  uninorm, $U\in \mathcal{U}.$ If $a$ is an idempotent element of $U$ then $u_a$ is either continuous or it is non-continuous in point $b,$ such that $b$ is an idempotent element of $U.$ \label{idid}
\end{proposition}

\begin{proof}
If $a$ is an idempotent element and $u_a$ is non-continuous in $b$ then Proposition \ref{atmost} implies that $U(a,x)<e$ for $x<b$ and   $U(a,x)>e$ for $x>b.$
If $a=e$ then $u_a$ is evidently continuous, thus we will  assume $a<e$ (the case for $a>e$ is analogous). Then since $U\in \mathcal{U}$ we have $b\geq e.$ If $b$ is not an idempotent element then
there exists $b_1$ with $e<b_1<b$ such that $U(b_1,b_1)>b$ and since $U$ is internal on $\{a\}\times \uint$ we have $$a=U(a,b_1)=U(a,b_1,b_1)=U(b_1,b_1)>b$$ what is a contradiction.
\end{proof}

\begin{proposition}
Let $U\colon \uint^2\lran \uint$ be a  uninorm $U\in \mathcal{U}.$ If
for some $x_1<x_2$ the function
 $u_{x_1}$ is non-continuous in $y$ and the function   $u_{x_2}$ is non-continuous in $y$ then $y$ is an idempotent element of $U.$ \label{proidc}
 \end{proposition}

\begin{proof} Assume that $y$ is not an idempotent element.
Since $U\in \mathcal{U}$ we have either $x_1<x_2\leq e$ or $e\leq x_1<x_2.$ We will suppose $x_1<x_2\leq e$  as the other case is analogous. Then also for all $f\in \cint{x_1,x_2}$ the function $u_f$ is non-continuous in $y$ and thus if there is an idempotent in $\cint{x_1,x_2}$ Proposition \ref{idid} implies that $y$ is an idempotent element. Assume the opposite and let $a,b\in \cint{0,e}$ be the idempotent elements such that $x_1,x_2 \in \opint{a,b}$ and there is no idempotent element in $\opint{a,b}.$
 Thus $u_a$ is non-continuous in point $y_1$ and $u_b$ in $y_2$ with $y_1>y>y_2,$
where $y_1$ and $y_2$ are idempotents.

 Now let $c,d\in \cint{e,1}$ be the idempotent elements such that $y\in \opint{c,d}$
and there is no idempotent element in $\opint{c,d}.$  Then $ y_2\leq c<d\leq y_1.$
Since according to Proposition \ref{restr} $(\lint{a,b} \cup \{U(b,c)\}\cup \rint{c,d})^2$ is closed under $U$ we can use a backward transformation $f^{-1}$ to the transformation given in \eqref{utra}. Thus we obtain a uninorm $U^*$ with continuous Archimedean underlying t-norm and t-conorm and there exist $v_1,v_2\in \cint{0,e},$ $v_1<v_2$ and $w\in \opint{e,1}$ such that either $U^*(v_1,w)=U^*(v_2,w)=e$ or $u^*_{v_1}$ and $u^*_{v_2}$ are non-continuous in $w.$ However, Propositions \ref{sumstrp} and \ref{pron} imply that $u^*_{v_1}$ and $u^*_{v_2}$ are continuous on $\opint{e,1}$ and thus we get $U^*(v_1,w)=U^*(v_2,w)=e$ what is not possible since then $v_1=U^*(v_1,v_2,w)=v_2$ what is a contradiction.
 \end{proof}

We again recall two results from \cite{mulfun}.

\begin{lemma}
Let $U\colon \uint^2\lran \uint$ be a uninorm, $U\in \mathcal{U}.$ Let $u_{x}$ be non-continuous in $y_1$ and $u_{y_2}$ be non-continuous in $x$ for some $y_1\neq y_2.$ Then for all $y\in \rint{y_1,y_2}$ ($y\in \lint{y_2,y_1}$) the function $u_y$ is non-continuous in $x.$ \label{lemconv}
\end{lemma}

\begin{lemma}
Let $U\colon \uint^2\lran \uint$ be a uninorm, $U\in \mathcal{U}.$ Assume  $x<e$ ($x>e$) such that  $u_{x}$  is continuous on $\uint$ and let $u_{y}$ be non-continuous in $x.$ Then for all $q\in \cint{y,1}$  ($q\in \cint{0,y}$) the function $u_q$ is non-continuous in $x.$ \label{lemconv2}
\end{lemma}

Now we can show the following.

\begin{corollary}
Let $U\colon \uint^2\lran \uint$ be a  uninorm, $U\in \mathcal{U},$ and let $r$ be its characterizing multi-function. Then border points of all types of maximal segments of $r$ are idempotent points.   \label{corid}  \end{corollary}

\begin{proof} Let the interval $\cint{a,b}$ correspond to some maximal horizontal segment of the characterizing multi-function $r$ of $U,$ i.e., for some $y\in \uint$ we have $r(x)=\{y\}$  for all $x\in \opint{a,b}.$ Also $y\notin r(x)$ for all
$x\in \uint\setminus \cint{a,b}.$ The Proposition \ref{proidc} implies that $y$ is an idempotent element.
If $u_y$ is continuous then Lemma \ref{lemconv} implies either $a=0$ or $b=1.$
If $u_y$ is non-continuous then Lemma \ref{lemconv} implies that $u_y$ is non-continuous either in $a$ or in $b,$
i.e., in all cases at least one of $a$ and $b$ is an idempotent element.

We will suppose $y\geq e$ (the case for $y\leq e$ is analogous). Then $a<b\leq e.$

If $a$ ($b$) is not an idempotent point
then taking the idempotent points $a_1,b_1,c_1$ such that $a\in \opint{a_1,b_1}$ ($b\in \opint{a_1,b_1}$)
and $c_1>y$ ($c_1<y$) and there is no idempotent point in $\opint{a_1,b_1}$ and in  $\opint{y,c_1}$ ($\opint{c_1,y}$)
the uninorm $U$ on $(\lint{a_1,b_1} \cup \{U(b_1,y)\} \cup \rint{y,c_1})^2$ (on $(\lint{a_1,b_1} \cup \{U(b_1,c_1)\} \cup \rint{c_1,y})^2$) is isomorphic with some uninorm $U^*$ with continuous Archimedean underlying t-norm and t-conorm such that there exists a point
$q\in \opint{0,e}$ with $U^*(q,p)>e$ for all $p>e$ ($U^*(q,p)<e$ for all $p<1$). Due to Propositions \ref{sumstrp} and \ref{pron} we get $U^*(x,y)=\max(x,y)$ ($U^*(x,y)=\min(x,y)$) for all $x\in \opint{0,e},$ $y\in \opint{e,1},$
however, then for all $a^*\in \opint{a_1,a}$ ( $b^*\in \opint{b,b_1}$) the function $u_{a^*}$ ($u_{b^*}$) is non-continuous in $y$ and we get $a=a_1$ ($b=b_1$). Thus both $a$ and $b$ are idempotent points.

Due to the symmetry of  the characterizing multi-function $r$ we can obtain the similar result also for vertical segments and then, since the graph of the  characterizing multi-function $r$ is a connected line, necessarily border points for all
types of segments are idempotents.

\end{proof}

\begin{lemma}
Let $U\colon \uint^2\lran \uint$ be a  uninorm, $U\in \mathcal{U},$ and let $a,b\in \uint,$ $a<b$ be idempotent points such that there is no idempotent in $\opint{a,b}.$ Then if there is a $y\in \uint$ such that the functions $u_{x_1}$ and $u_{x_2}$  for some $x_1,x_2\in\opint{a,b},$
$x_1<x_2,$ are non-continuous in $y$ then $U$ is pseudo-internal  on $\cint{a,b}\times \uint.$ \label{lemhoin}
\end{lemma}

\begin{proof} Assume $a<b\leq e$ (the case when $e\leq a<b$ is analogous). Then we want to show that $U$ is internal on $\cint{a,b}\times \cint{e,1}.$ In the opposite case Propositions \ref{sumstrp} and \ref{pron} imply that  there exist idempotent points $c,d$ such that on $(\lint{a,b} \cup \{U(b,c)\}\cup \rint{c,d})^2$  the uninorm $U$ is linearly isomorphic to a representable uninorm. Then, however, the characterizing multi-function of $U$ is strictly decreasing on $\opint{a,b}$ what is a contradiction.
\end{proof}

\begin{theorem}
Let $U\colon \uint^2\lran \uint$ be a  uninorm, $U\in \mathcal{U},$
where $e$ is the neutral element of $U.$
  Then  $U=V^{e},$ where $V^{e}=(\langle a_k,b_k,c_k,d_k,U_k \rangle \mid k\in K)^e$ is an extended ordinal sum of uninorms for some systems $(\opint{a_k,b_k})_{k\in K}$ and $(\opint{c_k,d_k})_{k\in K}$ satisfying all conditions of Proposition \ref{proext}, where
   for all $k\in K$ the uninorm $U_k$ is either internal (including the minimum t-norm and the maximum t-conorm), or representable (including continuous Archimedean t-norms  and t-conorms).
\end{theorem}
The proof of this result can be found in Appendix.

\begin{remark}
The previous result shows that each uninorm $U\in \mathcal{U}$ can be decomposed into ordinal sum of semigroups in the sense of Clifford, where the finest decomposition is such where all irreducible idempotent elements $b,$  i.e., such that $U(x,y)=b$ implies $b\in \{x,y\},$ are supports of a separate (one element) semigroups.  The remaining semigroups then contain at most one idempotent element. If $S$ is a subsemigroup of $(\uint,U)$ then  a similar result can be obtained also for all semigroups isomorphic with $S.$

\end{remark}

\section{Conclusions}

Each continuous t-norm (t-conorm) is equal to an ordinal sum of continuous Archimedean t-norms (t-conorms). In this paper we have extended this characterization onto uninorms with continuous  underlying t-norm and t-conorm. Using the characterizing multi-function we have shown that such a uninorm  can be decomposed into extended ordinal sum of representable uninorms, continuous Archimedean t-norms, continuous Archimedean t-conorms and internal uninorms. This result together with the properties of the characterizing multi-function offer a complete characterization of  uninorms from $\mathcal{U},$ i.e., of uninorms with continuous underlying t-norm and t-conorm.  The applications of these results are expected in all domains where uninorms are used.

 \vskip0.3cm \noindent{\bf \large Acknowledgement} \hskip0.3cm  This work was supported by grant  VEGA 2/0049/14, APVV-0178-11,
 and Program Fellowship of SAS.

\section{Appendix}

{\bf Proof of Proposition $11.$}

\begin{proof} First let us note that the ordinal sum construction ensures that if $a$ is an idempotent element of $U^e$  ($V^e$) then $U^e$ ($V^e$) is internal on $\{a\} \times \uint.$
The commutativity of $V^e$ is obvious. Further, $G\cap H = \emptyset$ and for $x\in G$ we have $g(x)=\lint{0,q_x}$ or $g(x)=\cint{0,q_x}$ for some idempotent point $q_x\geq e,$ and for $y\in H$ we have $h(y)=\lint{0,q_y}$ or $h(y)=\cint{0,q_y}$ for some idempotent point $q_y\leq e.$
Since $e\notin G\cup H$ we have $V^e(e,x)=x$ if $x\in \uint\setminus (G\cup H).$  If $x\in G$ then $e\in g(x)$ if $q_x>e,$ i.e., $V^e(e,x)=\min(x,e)=x.$
Assume $q_e=e.$ Then $g(x)=\cint{0,e}$ (since $c_k\geq e$ for all $k\in K$) and we have again $V^e(e,x)=\min(x,e)=x.$
If $y\in H$ then if $q_y<e$ we have $e\notin h(y)$ and  $V^e(e,y)=\max(y,e)=y.$ Now assume $q_y=e.$ Then $h(y)=\lint{0,e}$ (since $b_k\leq e$ for all $k\in K$), i.e., again
$e\notin h(y)$ and  $V^e(e,y)=\max(y,e)=y.$ Thus $e$ is a neutral element of  $V^e.$

\

In order to show that $V^e$ is non-decreasing in both coordinates it is sufficient to show that all $v_x$ for $x\in \uint,$ where
$v_x(y)=V^e(x,y)$ are non-decreasing. Assume any $x\in \uint.$ If $x=e$ then $v_x$ is evidently non-decreasing. Now assume $x<e$ (the case when $x>e$ is analogous). If $x\in G$ then $v_x(y)=\min(x,y)$ if $y\in g(x),$ where $g(x)=\lint{0,q_x}$ or $g(x)=\cint{0,q_x}$ for some idempotent point $q_x$ and  $v_x(y)=\max(x,y)$ otherwise, i.e., $v_x$ is evidently non-decreasing.
  Assume $x\notin G.$ If $U^e(x,y)= V^e(x,y)$ for all $y\in \uint$ then $v_x$ is evidently non-decreasing. Suppose $U^e(x,y)\neq V^e(x,y)$ for some $y\in \cint{0,e}.$ Then $y\in H,$ however, if we denote $A_1=\inf\{a_k\mid c_k=y\}$ and $A_2=\sup\{b_k\mid c_k=y\}$
  then  $U^e(x,y)\neq V^e(x,y)$ implies $x\in \lint{A_1,A_2}.$ If  $x\in \lint{a_k,b_k}$ for some $k\in K$ then $U^e(x,y)=x$ and $U^e(x,z)=z$ for all $z>y,$ i.e.,  $x=U^e(x,y)\neq V^e(x,y)=y$ will not violate the monotonicity of $v_x.$ If there exist $k_1,k_2\in H_y$ such that $b_{k_1}\leq x \leq b_{k_2}$ then the ordinal sum construction implies again $U^e(x,y)=x$ and $U^e(x,z)=z$ for all $z>y,$ i.e.,
  the monotonicity of $v_x$ is not violated. Finally assume $x=A_1,$ $A_1\notin \{a_i\}_{i\in K}.$ Then $A_1\in H_x^*$ and
  $U^e(x,z)=x$ for all $z<y,$ and $U^e(x,z)=z$ for all $z>y,$ i.e.,
  the monotonicity of $v_x$ is not violated.

\

Associativity: assume $x,y,z\in \uint ,$  if $V^e$ is internal on $\{x,y,z\}^2$ then Corollary \ref{lemass} implies the associativity.
Since all points  $q\in G\cup H$ are idempotent, i.e., $V^e$  is internal on $\{q\}\times\uint,$ we have only to check the case when at least one of $x,y,z$ does not belong to $G\cup H.$ If $x,y,z \in \uint\setminus (G\cup H)$ then the associativity follows from the associativity of $U^e.$
Further, if exactly one of $x,y,z$ belongs to $\uint\setminus (G\cup H)$ then this element differs from the two others and similarly as in Corollary \ref{lemass}
we can show the associativity. Thus the only remaining case is when exactly two of $x,y,z$ belong to $\uint\setminus (G\cup H).$

Assume  $x,y\in \uint\setminus (G\cup H),$ $x\leq y$ and $z \in G\cup H.$ If $V^e(x,y)\in \{x,y\}$ then the associativity can be shown as above.
Thus suppose $V^e(x,y)\notin \{x,y\}.$ We have either $V^e(x,y)<\min(x,y),$ or $V^e(x,y)>\max(x,y),$ or  $V^e(x,y)\in \opint{x,y}.$ Since the cases when $V^e(x,y)<\min(x,y)$ and when $V^e(x,y)>\max(x,y)$ are analogous we will show only the first one, i.e., if $V^e(x,y)<\min(x,y).$ Then
$x,y\in \opint{0,e}$ and if $z\in \cint{0,e}$ the associativity follows from the associativity of $U^e.$ Suppose
$z\in \rint{e,1}.$ Since $V^e(x,y)<\min(x,y)$ there exist idempotent points $a,b\in \cint{0,e}$ such that $x,y\in \opint{a,b}$ and $U(w,w)<w$ for all $w\in \opint{a,b}.$
If $U^e(x,z)=V^e(x,z)$ and  $U^e(y,z)=V^e(y,z)$ then the associativity follows from the associativity of $U^e.$ In the other case we have
$V^e(x,z)=V^e(y,z)=z$ and then $V^e(V^e(x,z),y)=V^e(z,y)=z=V^e(z,x) = V^e(V^e(y,z),x).$
If $V^e(x,y)\in \opint{a,b}$  the extended ordinal sum construction implies also $V^e(V^e(x,y),z)=z.$ If $V^e(x,y)=a$ then $U^e$ on $\cint{a,b}^2$ is a nilpotent t-norm and then again $V^e(V^e(x,y),z)=V^e(a,z)=z.$

Now suppose $V^e(x,y)\in \opint{x,y},$ i.e., $x\leq e\leq y.$ We will assume $z\in H$ (the case when $z\in G$ is analogical).
Here $U^e(x,y)=V^e(x,y)$ and  $U^e(z,y)=V^e(z,y).$ If $U^e(x,z)=V^e(x,z)$ the associativity follows from the associativity of $U^e.$ Thus suppose
$x=U^e(x,z)\neq V^e(x,z)=z.$ Then $y\geq z$ implies $V^e(x,y)=y$ what is a contradiction. However, for $e\leq y<z$ we have
$x=U^e(x,y)= V^e(x,y)$ which is again a contradiction.

\

Continuity of $T_{V^e}$ and $C_{V^e}$: since $U^e(x,y)\neq V^e(x,y)$ only on $\lint{0,e}\times \rint{e,1} \cup \rint{e,1}\times \lint{0,e}$
we have $T_{V^e}=T_{U^e}$ and $C_{V^e}=C_{U^e},$ i.e.,  $V^e\in \mathcal{U}$ if and only if   $U^e\in \mathcal{U}.$

\end{proof}

{\bf Proof of Theorem $4.$}

\begin{proof}
First we will identify the points $a_k,b_k,c_k,d_k$ for all $k\in K.$
Since $U\in \mathcal{U}$ we can assume the characterizing multi-function $r$  of $U.$ Recall that the graph of this function can be divided into connected maximal segments which are either strictly decreasing, or the horizontal segments, or the vertical segments (see Remark \ref{remseg}).  Due to Corollary \ref{corid} we know that borders of each horizontal and vertical segments are idempotent points, i.e., division is done in idempotent points.

Further since $T_U$ and $C_U$ are continuous the set of idempotent elements $I_U$ of $U$ is closed and thus $\cint{0,e}\setminus I_U = \bigcup\limits_{m\in M} \opint{a_m,b_m}$ and
$\cint{e,1}\setminus I_U = \bigcup\limits_{l\in L} \opint{c_l,d_l}$ for some countable index sets $M,L$ and two systems of open non-empty disjoint intervals $(\opint{a_m,b_m})_{m\in M}$
and $(\opint{c_l,d_l})_{l\in L}.$
 From the previous discussion it is evident that on each interval $\opint{a_m,b_m}$ for $m\in M$ ($\opint{c_l,d_l}$ for $l\in L$) the multi-function $r$ is either strictly decreasing, or it is a horizontal segment.
 Since the infinite union of closed sets need not to be closed, we define
  $B^*=\overline{\bigcup\limits_{m\in M} \cint{a_m,b_m}}\setminus \bigcup\limits_{m\in M} \cint{a_m,b_m},$ where $\overline{S}$ is the closure of the set $S,$  and
  $C^*=\overline{\bigcup\limits_{l\in L} \cint{c_l,d_l}}\setminus \bigcup\limits_{l\in L} \cint{c_l,d_l}.$ Since  $M$ and $L$ are countable also
  $B^*$ and $C^*$ are countable.
   If $\bigcup\limits_{m\in M} \cint{a_m,b_m}\cup B^*=\cint{0,e}$ and $\bigcup\limits_{l\in L} \cint{c_l,d_l}\cup C^*=\cint{e,1}$ we do not need any further preparations. In the opposite case we have
$\cint{0,e}\setminus (B^*\cup \bigcup\limits_{m\in M} \cint{a_m,b_m}) = \bigcup\limits_{o\in O^*} \opint{a_o,b_o},$ where   $(\opint{a_o,b_o})_{o\in O^*}$ is a system of non-empty disjoint open intervals, i.e., $O^*$ is a countable index set and $\cint{e,1}\setminus (C^*\cup \bigcup\limits_{l\in L} \cint{c_l,d_l}) = \bigcup\limits_{q\in Q^*} \opint{c_q,d_q},$ where   $(\opint{c_q,d_q})_{q\in Q^*}$ is a system of non-empty disjoint open intervals, i.e., $Q^*$ is a countable index set. We denote $B=B^* \setminus \{a_o \mid o\in O^*\}$ and $C=C^* \setminus \{d_q \mid q\in Q^*\}.$
We will further divide each open interval  $\opint{a_o,b_o}$ for $o\in O^*$ ($\opint{c_q,d_q}$ for $q\in Q^*$) into maximal open intervals that
are with respect to a multi-function $r$ either strictly decreasing, or horizontal segments, thus obtaining a new system of open disjoint intervals  $(\opint{a_o,b_o})_{o\in O}$ ($(\opint{c_q,d_q})_{q\in Q}$). Then $\bigcup\limits_{m\in M} \cint{a_m,b_m} \cup
\bigcup\limits_{o\in O} \cint{a_o,b_o} \cup B=\cint{0,e}$ and $\bigcup\limits_{l\in L} \cint{c_l,d_l} \cup \bigcup\limits_{q\in Q} \cint{c_q,d_q} \cup C =\cint{e,1}.$ Note that we can (and will) select the index sets $M,O,L,Q$ in such a way that they are mutually disjoint.

We will now define four disjoint index sets $K_1,\ldots,K_4$ for division of $\cint{0,e}$ (and similarly we will define $N_1,\ldots,N_4$ for $\cint{e,1}$) and four systems of open disjoint subintervals of $\cint{0,e}$ which are $(\opint{a_k,b_k})_{k\in K^*}$ for $K^*=K_1,\ldots,K_4,$ where
\begin{mylist}
\item $K_1\subseteq M$ and for $k\in K_1$ the multi-function $r$ is strictly decreasing on $\opint{a_k,b_k},$
\item $K_2\subseteq M$ and for $k\in K_2$ the multi-function $r$ is a horizontal segment  on $\cint{a_k,b_k},$
\item $K_3\subseteq O$ and for $k\in K_3$ the multi-function $r$ is strictly decreasing on $\opint{a_k,b_k},$
\item $K_4\subseteq O$ and for $k\in K_4$ the multi-function $r$ is a horizontal segment  on $\cint{a_k,b_k}.$
\end{mylist}

Now we can define summands: for $k\in K_1\cup K_3$ we obtain summands on supports of the type $\lint{a_k,b_k} \cup \rint{c,d},$
where  $\lim\limits_{x\lran a^+}r(x)=d$ and  $\lim\limits_{x\lran b^-}r(x)=c.$
Note that for $k\in K_1$ properties of $r$ ensures that there is no idempotent in $\opint{c,d}$ (see \cite{repord}) and since $c,d$ are idempotents then on $(\lint{a_k,b_k} \cup \{U(b_k,c)\} \cup \rint{c,d})^2$ the uninorm $U$ is isomorphic to a representable uninorm. If  $k\in K_3$
then on $(\lint{a_k,b_k} \cup \{U(b_k,c)\} \cup \rint{c,d})^2$ the uninorm $U$ is isomorphic to an s-internal uninorm (see Lemma \ref{lemmm}). Thus summands for $k\in K_1$ can be paired with summands for $n\in N_1$ and summands for $k\in K_3$ can be paired with summands for $n\in N_3.$ We will now relabel
each $n\in N_1$ to corresponding $k\in K_1$ and each $n\in N_3$ to corresponding $k\in K_3.$

 For  $k\in K_2\cup K_4$
the multi-function $r$ is a horizontal segment  on $\cint{a_k,b_k},$ i.e., $r(x)=\{y\}$ for all
$x\in \opint{a_k,b_k}$ and some $y\in \uint,$ and  we obtain summands on supports of the type $\lint{a_k,b_k} \cup \rint{y,y}.$
Similarly, we define summands  for $n\in N_2\cup N_4,$ i.e., we obtain summands on supports of the type $\lint{y,y} \cup \rint{c_n,d_n}.$

Further we will assume index sets $I,J$ such that $I,J,M,O,L,Q$ are mutually disjoint and $B=\{b_k\}_{k\in I}$ and
$C=\{c_k\}_{k\in J},$ i.e., $I$ and $J$ are countable. The definition of the set $B$ ($C$) implies that for each $k\in I$ ($k\in J$) there exists a strictly monotone sequence
$\{b_{m_i}\}_{i\in \mathbb{N}},$ where   $ m_i\in M$ ($\{c_{l_i}\}_{i\in \mathbb{N}},$ where $l_i\in L$) which converges to $b_k$ ($c_k$). For
$c_k=\lim\limits_{n\lran \infty} c_{m_i} $  ($b_k=\lim\limits_{n\lran \infty} b_{l_i} $) we define a summand on $\lint{b_k,b_k}\cup \rint{c_k,c_k}.$ In the case that for $k\in I$ ($k\in J$) we have  $c_k\in C$ ($b_k\in B$) we can pair these two points, and thus we obtain three mutually disjoint index sets,
$I_1,J_1,IJ$ such that $I=I_1\cup IJ$ and $J=J_1\cup IJ^*,$ where for each $k\in IJ$ there exists $m\in IJ^*$ such that $m$ and $k$ are paired, and vice-versa.

Now if we put $K=K_1\cup K_2 \cup K_3 \cup K_4 \cup N_2 \cup N_4 \cup I_1 \cup J_1 \cup IJ$ we will obtain all summands for the extended ordinal sum $V^e.$ We now relabel all relevant points in such a way that we obtain summands on supports $\lint{a_k,b_k} \cup \rint{c_k,d_k}$ for $k\in K.$ The properties of the multi-function $r$ ensures that the two systems $(\opint{a_k,b_k})_{k\in K}$  and   $(\opint{c_k,d_k})_{k\in K}$ are anti-comonotone and fulfill all required properties.

Similarly as in Proposition \ref{proext} we define sets $G$ and $H.$ Note that $G\cup H$ is a subset of the range of values of $r$ on the horizontal segments of the characterizing multi-function $r.$

We will show that for any $x\in \uint$ we have $U(x,y)=V^e(x,y)$ for all $y\in \uint.$
The extended ordinal sum construction ensures that $U=V^e$ on $\cint{0,e}^2$ and on $\cint{e,1}^2.$

Now we fix a $x\in \uint.$
If $x=e$ then evidently  $U(x,y)=V^e(x,y)$ for all $y\in \uint.$ From now on we will suppose $x<e$ (the case when $x>e$ is analogous).
We have to show $U(x,y)=V^e(x,y)$ for all $y\in \rint{e,1}.$
First suppose that $x\in G$ or  $y\in  H.$  As these cases are analogous we will suppose $x\in G.$
Then $\mathrm{Card}(r(x))>1$ and thus $r(x)=\cint{y_1,y_2}$ for some $y_1,y_2\in \cint{e,1}.$
Then  due to Corollary \ref{corid}  points $y_1$ and $y_2$ are idempotent and $u_x$ is either continuous or non-continuous in $y\in\cint{y_1,y_2}.$
If $u_x$ is non-continuous in $y\in\rint{y_1,y_2}$ then $U(x,z)=V^e(x,z)$ for all $z\in \uint,$ where $g(x)=\lint{0,y}$ or
$g(x)=\cint{0,y}$ (see Proposition \ref{proext}). Suppose $y=y_1.$ If $U(x,y_1)=\max(x,y_1)$ then $x\notin G.$ Thus  $U(x,y_1)=\min(x,y_1)$ and again
$U(x,z)=V^e(x,z)$ for all $z\in \uint,$ where $g(x)=\cint{0,y_1}.$ If $u_x$ is continuous then Lemma \ref{lemconv2} implies
$y_2=1$ and thus  $U(x,z)=V^e(x,z)$ for all $z\in \uint,$ where $g(x)=\lint{0,1}$ or
$g(x)=\cint{0,1}.$

Further we will assume $x\notin G$ and $y\notin H$ and we will divide the proof into several cases:

\noindent Case $1$:  if $\mathrm{Card}(r(x))>1,$ i.e., $r(x)=\cint{y_1,y_2}$ for some $y_1,y_2\in \cint{e,1}.$
 Since $x\notin G$ we have $U(x,z)=V^e(x,z)$ for all $z\in \uint,$ where $g(x)=\lint{0,y_1}.$

\noindent Case $2$: if $\mathrm{Card}(r(x))=1$ and $x\in \opint{a_k,b_k}$ for some $k\in K.$

 If $k\in K_1$ we set  $\lim\limits_{x\lran a_k^+}r(x)=d$ and  $\lim\limits_{x\lran b_k^-}r(x)=c.$ Then on $(\lint{a_k,b_k} \cup \{U(b_k,c)\} \cup \rint{c,d})^2$ the uninorm $U$  is a transformation of a representable uninorm given by \eqref{unitr}. The monotonicity, the neural element $e$ and the fact that $a_k,b_k,c,d$ are idempotent implies $U(x,y)=\min(x,y)=V^e(x,y)$
for $y\in \cint{b_k,c}.$ Further, since $x\in \opint{a_k,b_k}$ we have $U(d,x)=d=\max(d,x)$ and  the monotonicity gives
 $U(y,x)=\max(y,x)=V^e(x,y)$ for all  $y>d.$ If $k\in K_2$ then by Lemma \ref{lemhoin} the uninorm $U$ is pseudo-internal on $\opint{a_k,b_k}\times \uint,$ i.e.,
 if $r(x)=\{z\}$ for all $x\in \opint{a_k,b_k}$ we have $U(x,y)=\min(x,y)=V^e(x,y)$ for $y\in \opint{e,z}$ and $U(x,y)=\max(x,y)=V^e(x,y)$ for $y>z.$ The only question arises when $y=z,$ however, then $y \in H$ what was solved above. If $k\in K_3$ we again  set  $\lim\limits_{x\lran a_k^+}r(x)=d$ and  $\lim\limits_{x\lran b_k^-}r(x)=c.$ Then Proposition \ref{idid} implies that on $(\lint{a_k,b_k} \cup \{U(b_k,c)\} \cup \rint{c,d})^2$ the uninorm $U$  is a transformation of an s-internal uninorm. Further, monotonicity and strict monotonicity of $r$ on $\opint{a_k,b_k}$  implies that $U(x,y)=\min(x,y)=V^e(x,y)$
for $y\in \cint{b_k,c}$ and $U(x,y)=\max(x,y)=V^e(x,y)$ for $y>d.$ Finally, if  $k\in K_4$ and  $r(x)=\{z\}$ for all $x\in \opint{a_k,b_k},$ since
 $U$ is internal on $\opint{a_k,b_k}\times \uint,$ we have $U(x,y)=\min(x,y)=V^e(x,y)$ for $y\in \opint{e,z}$ and $U(x,y)=\max(x,y)=V^e(x,y)$ for $y>z.$ Again,
 $z\in  H$ and thus the case when $y=z$ was solved above. Summarising, if  $x\in \opint{a_k,b_k}$ for some $k\in K$
 and $y\notin  H$ then $U(x,y)=V^e(x,y).$

  \noindent Case $3$: if $\mathrm{Card}(r(x))=1$ and $x=a_k$ for some $k\in K.$
If $k\in K_1\cup K_3$  then $\mathrm{Card}(r(x))=1$ implies that
 $u_{a_k}$ is non-continuous in $d=\lim\limits_{x\lran a_k^+}r(x)$ and since $a_k$ is idempotent we have  $U(x,y)=V^e(x,y)$ for all $y\in \cint{e,1}\setminus H.$
 If $k\in K_2\cup K_4$ then $\mathrm{Card}(r(x))=1$ implies that $u_{a_k}$ is non-continuous in $z,$ such that $r(q)=\{z\}$ for all $q\in \opint{a_k,b_k},$ and  since $a_k$ is idempotent we have  $U(x,y)=V^e(x,y)$ for all $y\in \cint{e,1}\setminus H.$

\noindent Case $4$: if $\mathrm{Card}(r(x))=1$ and  $x=b_k$ for some $k\in K.$
 If $x=a_m$ for some   $m\in K$ the result follows from the previous case. In the other case
 $b_k\in B$ and the result follows from the monotonicity of $U$ and all previous cases.

\end{proof}


\begin{thebibliography}{99}
\bibitem{AFS06}  C. Alsina , M. J.  Frank, B. Schweizer (2006). Associative Functions: Triangular Norms and Copulas, World Scientific, Singapore.

\bibitem{cli} A. H. Clifford (1954). Naturally totally  ordered commutative semigroups. American Journal of Mathematics 76, pp.  631–-646.

\bibitem{ordut} P. Drygas (2010). On properties of uninorms with underlying t-norm and t-conorm
given as ordinal sums.  Fuzzy Sets and Systems 161(2), pp.  149--157.

\bibitem{FYR97} J. C. Fodor, R. R. Yager, A. Rybalov (1997). Structure of uninorms. International Journal of Uncertainty, Fuzziness and Knowledge-based Systems 5, pp. 411--127.

\bibitem{haj98} P. H\' ajek (1998). Metamathematics of fuzzy logic. Kluwer Academic Publishers, Dordrecht.

\bibitem{FB12}  J. Fodor, B. De Baets (2012).
A single-point characterization of representable uninorms.
Fuzzy Sets and Systems 202, pp.  89–-99.




\bibitem{KMP00} E. P. Klement, R. Mesiar, E. Pap (2000). Triangular norms. Kluwer Academic Publishers, Dordrecht.

\bibitem{KD69} R. L. Kruse, J. J. Deely (1969).
Joint continuity of monotonic functions.
The American Mathematical Monthly 76(1), pp. 74--76.

\bibitem{PA2} W. Pedrycz and K. Hirota (2007). Uninorm-based logic neurons as adaptive
and interpretable processing constructs. Soft Computing 11(1),
pp. 41–-52.


\bibitem{Lili}
G. Li, H.W. Liu, Distributivity and conditional distributivity of a weakly
continuous uninorm over a continuous t-conorm, Fuzzy Sets and Systems,
accepted.

\bibitem{Lilif} G. Li, H.W. Liu, J. Fodor (2014). Single-point characterization of uninorms
with nilpotent underlying t-norm and t-conorm. International Journal of
Uncertainty, Fuzziness and Knowledge-Based Systems 22, pp. 591--604.

\bibitem{PT1} H.-W. Liu (2013).
Fuzzy implications derived from generalized
additive generators of representable uninorms.
IEEE Transactions on Fuzzy Systems 21(3), pp.   555--566.



\bibitem{MesPet14} M. Petr\'\i k, R. Mesiar (2014). On the structure of special classes of uninorms.
Fuzzy Sets and Systems 240, pp. 22--38.

\bibitem{MMT}  J. Mart\'\i n, G. Mayor and J. Torrens (2003).  On locally internal monotonic operations. Fuzzy Sets and Systems 137, pp. 27--42.

\bibitem{jacon}  A. Mesiarov\' a-Zem\' ankov\' a. Multi-polar t-conorms and uninorms, Information Sciences,doi: 10.1016/j.ins.2014.12.060 .

\bibitem{genuni}  A. Mesiarov\' a-Zem\' ankov\' a. Generated generalized uninorms and ordinal sum of uninorms,
Int. J. Approximate Reasoning, under review.

\bibitem{repord} A. Mesiarov\' a-Zem\' ankov\' a. Ordinal sums of representable uninorms.
Fuzzy Sets and Systems, submitted.

\bibitem{mulfun}  A. Mesiarov\' a-Zem\' ankov\' a.  Characterization of uninorms with continuous underlying t-norm and t-conorm by their set of  discontinuity points, IEEE Transactions on Fuzzy Systems, under review.

%\bibitem{Pap95} E. Pap (1995). Null-Additive Set Functions. Dordrecht, the Netherlands, Kluwer.


\bibitem{uniru}   D. Ruiz, J. Torrens (2006). Distributivity and conditional distributivity of a
uninorm and a continuous t-conorm. IEEE Transactions on Fuzzy Systems 14(2), pp. 180--190.

\bibitem{Rut14}
D. Ruiz-Aguilera, J. Torrens (2014). A characterziation of discrete uninorms having smooth
underlying operators, Fuzzy Sets and Systems, doi: 10.1016/j.fss.2014.10.020


\bibitem{16} B. Schweizer, A. Sklar (1983). Probabilistic Metric Spaces, North-Holland, New York.


\bibitem{sug85}  M. Sugeno (1985). Industrial applications of fuzzy control. New York, Elsevier.


\bibitem{YR96} R. R. Yager, A. Rybalov (1996). Uninorm aggregation operators. Fuzzy Sets and Systems
80, pp. 111–-120.

\bibitem{PA1}  R. R. Yager (2001).
Uninorms in fuzzy systems modeling. Fuzzy Sets and
Systems 122(1), pp. 167-–175.

%\bibitem{Yag05}  R. R. Yager (2005).
% Extending multicriteria decision making by
%mixing t-norms and OWA operators. International Journal of Intelligent Systems 20, pp. 453-–474.



\end{thebibliography}
\end{document}